\RequirePackage{fix-cm}
\documentclass{amsart} 
\usepackage[utf8]{inputenc}
\usepackage{mathptmx}
\usepackage{latexsym}
\usepackage{amsmath}
\usepackage{amssymb}
\usepackage{amsfonts} 
\usepackage{eucal} 
\usepackage{amsbsy}
\usepackage{amsthm}
\usepackage[all]{xy}
\usepackage[hyperindex,breaklinks,pagebackref=true]{hyperref}
\usepackage{cleveref}
\usepackage[boxsize=.8em]{ytableau}
\usepackage{graphicx}
\usepackage{framed}
\usepackage{tikz}

\newtheorem{thm}{Theorem}[section]
\newtheorem*{thm*}{Theorem}
\newtheorem{lemma}[thm]{Lemma}
\newtheorem{prop}[thm]{Proposition}
\newtheorem{cor}[thm]{Corollary}
\newtheorem*{cor*}{Corollary}

\theoremstyle{definition}
\newtheorem{defn}[thm]{Definition}
\newtheorem{example}[thm]{Example}

\theoremstyle{remark}
\newtheorem{remark}[thm]{Remark}

\newcommand {\Fa}    {\ensuremath{\mbox{$\mathcal{F}$}}}

\newcommand {\Ha}    {\ensuremath{\mbox{$\mathcal{H}$}}}

\newcommand {\intg}  {\ensuremath{\mathbb{Z}}}

\newcommand {\rat}   {\ensuremath{\mathbb{Q}}}

\newcommand {\rk}    {\operatorname{rk}}

\newcommand {\ch}    {\ensuremath{\operatorname{ch}}}

\newcommand {\colim} {\ensuremath{\operatorname{colim}}}

\newcommand {\BM}   {{\ensuremath{\operatorname{BM}}}}

\newcommand {\oO}    {\ensuremath{\mathcal{O}}}

\newcommand {\cl}   {{\ensuremath{\operatorname{cl}}}}

\newcommand{\td}   {\ensuremath{\operatorname{td}}}
\newcommand{\Per} {\ensuremath{\operatorname{Per}}}
\newcommand{\alg} {\ensuremath{\operatorname{alg}}}

\makeatletter
\@namedef{subjclassname@2020}{%
  \textup{2020} Mathematics Subject Classification}
\makeatother

\begin{document}


\title{Topological Gysin Coherence for Algebraic Characteristic Classes of Singular Spaces}

\author{Markus Banagl}

\address{Institut f\"{u}r Mathematik, Universit\"{a}t Heidelberg,
  Im Neuenheimer Feld 205, 69120 Heidelberg, Germany}

\email{banagl@mathi.uni-heidelberg.de}

\author{J\"{o}rg Sch\"{u}rmann}

\address{Mathematisches Institut, Universit\"{a}t M\"{u}nster, 
Einsteinstr. 62, 48149 M\"{u}nster, Germany}
\email{jschuerm@uni-muenster.de}

\author{Dominik J. Wrazidlo}

\address{Institut f\"{u}r Mathematik, Universit\"{a}t Heidelberg,
  Im Neuenheimer Feld 205, 69120 Heidelberg, Germany}

\email{dwrazidlo@mathi.uni-heidelberg.de}

\thanks{M. Banagl and D. Wrazidlo are funded in part by the Deutsche Forschungsgemeinschaft (DFG, German Research Foundation) through a research grant to the first author (Projektnummer 495696766).
J. Sch\"{u}rmann is funded by the Deutsche Forschungsgemeinschaft (DFG, German Research Foundation) Project-ID 427320536 -- SFB~1442, as well as under Germany's Excellence Strategy EXC 2044 390685587, Mathematics M\"{u}nster: Dynamics -- Geometry -- Structure.}

\date{\today}

\subjclass[2020]{57R20, 55R12, 55N33, 57N80, 32S60, 32S20, 14M15, 14C17, 32S50}

\keywords{Gysin transfer, Characteristic Classes, Singularities, Stratified Spaces, Intersection Homology, Goresky-MacPherson $L$-classes, Chern classes, Todd classes, Verdier-Riemann-Roch formulae, Schubert varieties, Intersection theory, Transversality}


\begin{abstract}
Brasselet, the second author and Yokura introduced Hodge-theoretic Hirzebruch-type characteristic classes $IT_{1, \ast}$, and conjectured that they are equal to the Goresky-MacPherson $L$-classes for pure-dimensional compact complex algebraic varieties.
In this paper, we show that the framework of Gysin coherent characteristic classes of singular complex algebraic varieties developed by the first and third author in previous work applies to the characteristic classes $IT_{1, \ast}$.
In doing so, we prove the ambient version of the above conjecture for a certain class of subvarieties in a Grassmannian, including all Schubert subvarieties.
Since the homology of Schubert subvarieties injects into the homology of the ambient Grassmannian, this implies the conjecture for all Schubert varieties in a Grassmannian.
We also study other algebraic characteristic classes such as Chern classes and Todd classes (or their variants for the intersection cohomology sheaves) within the framework of Gysin coherent characteristic classes.
\end{abstract}

\maketitle


\tableofcontents


\section{Introduction}
We show that various algebraic characteristic classes for singular spaces fit into the framework of Gysin coherent characteristic classes of singular complex algebraic varieties that was developed by the first and the third author in \cite{bw}.
This framework formalizes Verdier-Riemann-Roch type formulae with respect to Gysin restriction arising from transverse intersection with smooth varieties.
In particular, we show that the framework applies to the Hodge-theoretic intersection Hirzebruch characteristic classes $IT_{1, \ast}$ introduced by the second author jointly with Brasselet and Yokura in \cite{bsy}.
The classes $IT_{1, \ast}$ are conjectured in \cite[Remark 5.4]{bsy} to be equal to the topological $L$-classes of Goresky and MacPherson for pure-dimensional compact complex algebraic varieties.
By applying the uniqueness theorem for Gysin coherent characteristic classes of \cite{bw} to both the classes $IT_{1, \ast}$ and the Goresky-MacPherson $L$-classes, we conclude that both classes coincide on Cohen-Macaulay subvarieties of Grassmannians after push-forward into the homology of the ambient Grassmannian.
Since the homology of Schubert subvarieties injects into the homology of an ambient Grassmannian, this implies the conjecture for all Schubert varieties in a Grassmannian.

The notion of Gysin coherent characteristic classes is recalled in \Cref{Gysin coherent characteristic classes}.
This notion was introduced in \cite{bw} with respect to a family $\mathcal{X}$ of inclusions $i \colon X \rightarrow W$ of compact irreducible subvarieties of smooth pure-dimensional projective complex algebraic varieties $W$ which satisfy primarily an analog of the Kleiman-Bertini transversality theorem with respect to a suitable notion of transversality.
A Gysin coherent characteristic class is a pair $c\ell = (c\ell^{\ast}, c\ell_{\ast})$ consisting of a function $c\ell^{\ast}$ that assigns to every inclusion $f \colon M \rightarrow W$ of a smooth closed subvariety $M \subset W$ in a smooth variety $W$ a normalized element $c\ell^{\ast}(f) \in H^{\ast}(M; \mathbb{Q})$, and a function $c\ell_{\ast}$ that assigns to every inclusion $i \colon X \rightarrow W$ of a compact possibly singular subvariety of a smooth variety an element $c\ell_{\ast}(i) \in H_{\ast}(W; \mathbb{Q})$ whose highest non-trivial homogeneous component is the ambient fundamental class of $X$ in $W$ such that the Gysin restriction formula
$$
f^{!} c\ell_{\ast}(i) = c\ell^{\ast}(f) \cap c\ell_{\ast}(M \cap X \subset M)
$$
holds for all $i$ contained in $\mathcal{X}$.
Here, $f^{!}$ denotes the topological Gysin homomorphism on singular rational homology.
Furthermore, we require that $c\ell_{\ast}$ is multiplicative under products, that $c\ell^{\ast}$ and $c\ell_{\ast}$ transform naturally under isomorphisms of ambient smooth varieties, and that $c\ell_{\ast}$ is natural with respect to inclusions in larger ambient smooth varieties.
It was shown in \cite{bw} that the Goresky-MacPherson $L$-class gives rise to a Gysin coherent characteristic class.

In this paper, the focus lies on application of the Gysin coherence framework to the Hodge-theoretic intersection Hirzebruch characteristic classes $IT_{1, \ast}$ introduced by the second author jointly with Brasselet and Yokura in \cite{bsy}.
\Cref{sec.hodgeclasses} provides a brief outline of the theory of Hodge-theoretic characteristic classes on singular algebraic varieties.
For an introduction to algebraic characteristic classes of singular spaces via mixed Hodge theory in the complex algebraic context, see the second author's expository paper \cite{schuermannmsri}.
For an introduction to topological characteristic classes of singular spaces, see \cite{banagltiss}.

Let $L^{\ast}(\nu)$ denote the Hirzebruch $L$-class of a topological vector bundle $\nu$.
The main result of the present paper is \Cref{proposition it class is l type characteristic class}, which can be stated as follows.

\begin{thm*}
The pair $\mathcal{L} = (\mathcal{L}^{\ast}, \mathcal{L}_{\ast})$ defined by $\mathcal{L}^{\ast}(f) = L^{\ast}(\nu_{f})$ for every inclusion $f \colon M \rightarrow W$ of a smooth closed subvariety $M \subset W$ in a smooth complex algebraic variety $W$ with normal bundle $\nu_{f}$, and by $\mathcal{L}_{\ast}(i) = i_{\ast} IT_{1, \ast}(X)$ for every inclusion $i \colon X \rightarrow W$ of a compact possibly singular subvariety $X \subset W$ in a smooth variety $W$ is a Gysin coherent characteristic class with respect to the family $\mathcal{X}_{CM}$ of inclusions $i \colon X \rightarrow W$ such that $X$ is Cohen-Macaulay.
\end{thm*}

In \Cref{review transverse setup}, we discuss the main ingredient for the proof, which is a Verdier-Riemann-Roch type formula derived by the first author in \cite{banagllgysin} for the behavior of the Hodge-theoretic intersection Hirzebruch characteristic classes $IT_{1 \ast}$ under algebraic Gysin restriction with respect to normally nonsingular embeddings of singular spaces.
In order for this formula to fit the axioms of Gysin coherent characteristic classes, we need to apply it in a transverse setup (see \Cref{theorem gysin restriction of hodge theoretic cc}), where Tor-independence turns out to be the correct notion of transversality for subvarieties of an ambient smooth variety.
Since Sierra's Kleiman-Bertini transversality theorem for Tor-independence \cite{sierra} is only available under additional assumptions, we need to incorporate the Cohen-Macaulay condition into our result.
For the Gysin restriction formula to be valid in the transverse setup, we need to handle some technical issues regarding the behavior of Whitney stratifications under blow-up of complex manifolds along transverse submanifolds, see \Cref{Blow-up and transversality}.
Finally, to establish the topological Gysin coherence axiom for the algebraic characteristic class $IT_{1 \ast}$, we compare the algebraic Gysin map with the topological Gysin map in \Cref{Topological and algebraic Gysin restriction}, and find that they coincide at least on algebraic cycles (see \Cref{alg and top gysin}).

Since the Goresky-MacPherson $L$-class is a Gysin coherent characteristic class, the uniqueness theorem for such classes implies that both the classes $IT_{1, \ast}$ and the Goresky-MacPherson $L$-classes coincide on irreducible Cohen-Macaulay subvarieties of Grassmannians after push-forward into the homology of the ambient Grassmannian (see \Cref{main result conjecture}).
Since Schubert varieties in a Grassmannian are Cohen-Macaulay, and their homology injects into the homology of the Grassmannian, we obtain \Cref{conjecture for schubert holds}, which states the following.

\begin{cor*}
The equality $L_{\ast}(X) = IT_{1, \ast}(X)$ holds for all Schubert varieties $X$ in a Grassmannian.
\end{cor*}

The above equality was conjectured more generally for pure-dimensional compact complex algebraic varieties in \cite[Remark 5.4]{bsy}.
In joint work with Cappell, Maxim, and Shaneson, the second author proved the conjecture in \cite[Cor. 1.2]{cmss1} for orbit spaces $X = Y/G$, with $Y$ a projective $G$-manifold and $G$ a finite group of algebraic automorphisms.
They also showed the conjecture for certain complex hypersurfaces with isolated singularities which are assumed rational homology manifolds \cite[Theorem 4.3]{cmss2}.
The conjecture holds for simplicial projective toric varieties as shown by Maxim and the second author \cite[Corollary 1.2(iii)]{ms}.
In \cite{banaglcovertransfer}, the first author proved that his extension of the Goresky-MacPherson $L$-class to oriented pseudomanifolds that possess Lagrangian structures along strata of odd codimension, introduced in \cite{banagl-lcl}, transfers to the $L$-class of any finite covering space (see Theorem 3.14 in \cite{banaglcovertransfer}).
He deduced from this that the conjecture holds for normal connected complex projective $3$-folds $X$ that have at worst canonical singularities, trivial canonical divisor, and $\operatorname{dim} H^{1}(X; \oO_{X}) > 0$.
If $\xi$ is an oriented PL block bundle over a closed PL Witt space $B$ (e.g. a pure-dimensional complex algebraic variety) with closed $d$-dimensional PL manifold fiber and total space $X$, then $\xi^{!} L_{\ast}(B) = L^{\ast}(\nu_{\xi}) \cap L_{\ast}(X)$ under the block bundle transfer $\xi^{!} \colon H_{\ast}(B; \mathbb{Q}) \rightarrow H_{\ast+d}(X; \mathbb{Q})$ associated to $\xi$, see \cite{banaglbundletrafer}.
Here, $\nu_{\xi}$ is the oriented stable vertical normal PL microbundle of $\xi$.
In fact, it is shown in \cite{banaglko} that the $\operatorname{KO}$-theoretic block bundle transfer $\xi^{!} \colon \operatorname{KO}_{\ast}(B) \otimes \mathbb{Z}[\frac{1}{2}] \rightarrow\operatorname{KO}_{\ast+d}(X) \otimes \mathbb{Z}[\frac{1}{2}]$ sends the Siegel-Sullivan orientation $\Delta(B)$ to $\Delta(X)$.
The Siegel-Goresky-MacPherson $L$-class can be recovered from $\Delta(-)$ by applying the Pontrjagin character.
Note that transfer does not generally commute with localization of homotopy theoretic spectra.
The aforementioned cases of the conjecture concern rational homology manifolds.
Fern\'{a}ndez de Bobadilla and Pallar\'{e}s \cite{fdbp} proved the conjecture for all projective complex algebraic varieties that are rational homology manifolds.
In joint work with Saito \cite{fdbps}, they also gave a different proof for compact instead of projective complex varieties, using the theory of mixed Hodge modules.
On the other hand, Schubert varieties in a Grassmannian are generally singular enough so as not to be rational homology manifolds.

Furthermore, we shall clarify how other algebraic characteristic classes such as Chern classes (see \Cref{example chern class}) and Todd classes (see \Cref{example todd class}) (or their variants for the intersection cohomology sheaves) fit into the framework of Gysin coherence.
In \Cref{proposition todd class is l type characteristic class}, we show that Todd classes are Gysin coherent with respect to the family $\mathcal{X}_{CM}$ of inclusions $X \rightarrow W$ such that $X$ is Cohen-Macaulay.
The proof is similar to that of our main result, but is based on the Verdier-Riemann-Roch formula for the Todd class transformation $\tau_{\ast}$, which was conjectured by Baum-Fulton-MacPherson in \cite[p. 137]{bfm}, and proved by Verdier \cite[p. 214, Theorem 7.1]{verdierintcompl}.
On the other hand, \Cref{thm chern classes are gysin coherent cc} shows that Chern classes are Gysin coherent, which exploits the second author's Verdier-Riemann-Roch type theorem (see \cite{schuertrans}) for the behavior of the Chow homology Chern class transformation $c_{\ast} \colon F(X) \rightarrow A_{\ast}(X)$ on complex algebraically constructible functions under refined Gysin maps associated to transverse intersections in a microlocal context.

We conclude by mentioning that the normally nonsingular expansions derived in \cite[Theorem 7.1]{bw} provide a systematic method for the recursive computation of Gysin coherent characteristic classes in ambient Grassmannians in terms of genera of explicitly constructed characteristic subvarieties.
For the classes discussed in this paper, the computational consequences e.g. in the case of Schubert varieties in a Grassmannian remain open for future study.
Chern classes of Schubert varieties in Grassmannians were computed by Aluffi and Mihalcea in \cite{aluffi}.
Moreover, an algorithm for the computation of Chern classes for Schubert cells in a generalized flag manifold is provided in \cite{am}.
For an extension to the equivariant setting, as well as applications to positivity of non-equivariant Chern classes and related classes, we refer the reader to \cite{amss}.

\textbf{Notation.}
Regarding singular cohomology, singular homology and Borel-Moore homology, we work with rational coefficients and will write these groups as $H^{\ast}(-)$, $H_{\ast}(-)$, and $H_{\ast}^{\BM}(-)$.

\section{Hodge-Theoretic Characteristic Classes}\label{sec.hodgeclasses}
This section provides a brief outline of the theory of Hodge-theoretic characteristic classes on singular algebraic varieties (compare also Section 5 in \cite{banagllgysin}).
After recalling the motivic Hodge Chern class transformation in \Cref{motivic Hodge Chern class transformation} and the twisted Todd transformation of Baum, Fulton, MacPherson in \Cref{definition twisted todd transformation}, we define the motivic Hirzebruch class transformation in \Cref{motivic hirzebruch class transformation}, and finally the intersection Hirzebruch characteristic class $IT_{y \ast}$ in \Cref{intersection generalized todd class}.
For a more detailed exposition of the topic, we refer e.g. to the expository paper \cite{schuermannmsri}.

For an algebraic variety $X$, let $K^{\alg}_{0} (X)$ denote the Grothendieck
group of the abelian category of coherent sheaves of $\oO_X$-modules.
When there is no danger of
confusion with other $K$-homology groups, we shall also write $K_0 (X) = K^{\alg}_0 (X)$.
Let $K^0 (X) = K_{\alg}^0 (X)$ denote the Grothendieck group of
the exact category of algebraic vector
bundles over $X$.
The tensor product $\otimes_{\oO_X}$ induces a
\emph{cap product}
\[ \cap: K^0 (X) \otimes K_0 (X) \longrightarrow K_0 (X),~
  [E] \cap [\Fa] = [E \otimes_{\oO_X} \Fa]. \]
Thus,
\begin{equation} \label{equ.kthtokhomalg}
-\cap [\oO_X]: K^0 (X)\longrightarrow K_0 (X)
\end{equation}
sends a vector bundle $[E]$ to its associated (locally free) sheaf of germs of local
sections $[E\otimes \oO_X]$.
If $X$ is smooth, then $-\cap [\oO_X]$ is an isomorphism.

Let $X$ be a complex algebraic variety and $E$ an algebraic vector bundle
over $X$. For a nonnegative integer $p$, let $\Lambda^p (E)$ denote the
$p$-th exterior power of $E$.
The \emph{total $\lambda$-class} of $E$ is by definition
\[ \lambda_y (E) = \sum_{p\geq 0} \Lambda^p (E)\cdot y^p, \]
where $y$ is an indeterminate functioning as a bookkeeping device.
This construction induces a homomorphism
$\lambda_y (-): K^0_{\alg} (X) \longrightarrow K^0_{\alg} (X)[[y]]$
from the additive group of $K^0 (X)$ to the multiplicative monoid of
the power series ring $K^0 (X)[[y]]$.
Now let $X$ be a smooth variety, let $TX$ denote its
holomorphic tangent bundle and
$T^* X$ its holomorphic cotangent bundle.
Then $\Lambda^p (T^* X)$ is the vector bundle of
holomorphic $p$-forms on $X$.
Its associated sheaf of sections is denoted by $\Omega^p_X$. Thus
\[ [\Lambda^p (T^* X)] \cap [\oO_X] = [\Omega^p_X] \]
and hence
\[ \lambda_y (T^* X) \cap [\oO_X] = \sum_{p=0}^{\dim X} [\Omega^p_X] y^p. \]

Let $X$ be a complex algebraic variety and let
$MHM (X)$ denote the abelian category of M. Saito's algebraic mixed
Hodge modules on $X$.
Totaro observed in \cite{totaro} that Saito's construction of a pure
Hodge structure on the intersection homology of compact varieties
implicitly contains a
definition of certain characteristic homology classes for singular
algebraic varieties.
The following definition is based on this observation and due to
Brasselet, Sch\"{u}rmann and Yokura, \cite{bsy}, see also the expository paper
\cite{schuermannmsri}.
\begin{defn}\label{motivic Hodge Chern class transformation}
The \emph{motivic Hodge Chern class transformation}
\[ MHC_y: K_0 (MHM(X)) \to K^{\alg}_0 (X) \otimes \intg [y^{\pm 1}] \]
is defined by
\[ MHC_y [M]
= \sum_{i,p} (-1)^i [\Ha^i (Gr^F_{-p} DR[M])] (-y)^p. \]
\end{defn}
Here,
$Gr^F_p DR: D^b MHM (X) \to
  D^b_{\operatorname{coh}} (X),$
with $D^b_{\operatorname{coh}} (X)$ the bounded derived category
of sheaves of $\oO_X$-modules with coherent cohomology sheaves,
denotes the functor of triangulated categories
constructed by M. Saito, see
\cite[\S 2.3]{saito88},
\cite[\S 1]{saito00},
\cite[\S 3.10]{saito90},
obtained by taking a suitable filtered de Rham complex of the filtered
holonomic $D$-module underlying a mixed Hodge module.
For every $p\in \intg$, these functors induce functors between the
associated Grothendieck groups, with $Gr^F_p DR[M] \simeq 0$ for a given $M$ and almost all $p$.

A flat morphism $f:X\to Y$ gives rise to a flat pullback
$f^*: \operatorname{Coh} (Y)\to \operatorname{Coh}(X)$
on coherent sheaves, which is exact and hence induces a flat pullback
$f^*_K: K^{\alg}_0 (Y)\to K^{\alg}_0 (X)$. This applies in particular
to smooth morphisms and is then often called smooth pullback.
An arbitrary algebraic morphism $f:X \to Y$
(not necessarily flat) induces a homomorphism
\[ f^*: K_0 (MHM(Y)) \longrightarrow K_0 (MHM (X)) \]
which corresponds under the forgetful functor
$\operatorname{rat}: D^b MHM (-)\to D^b_c (-;\rat)$
to $f^{-1}$ on constructible complexes of sheaves.
(Additional remarks on $\operatorname{rat}$ are to be found further below.)
We record Sch\"{u}rmann's \cite[Cor. 5.11, p. 459]{schuermannmsri}:
\begin{prop} \label{prop.mhcvrrsmpullb}
(Verdier-Riemann-Roch for smooth pullbacks.)
For a smooth morphism $f:X\to Y$ of complex algebraic varieties,
the Verdier Riemann-Roch formula
\[ \lambda_y (T^*_{X/Y}) \cap f^*_K MHC_y [M] =
  MHC_y (f^* [M]) = MHC_y [f^* M] \]
holds for $M \in D^b MHM (Y),$ where $T^*_{X/Y}$ denotes
the relative cotangent bundle of $f$.
\end{prop}

Let $E$ be a complex vector bundle and let $a_i$ denote the Chern roots of $E$.
In \cite{hirzebruch},
Hirzebruch introduced a cohomological characteristic class
\[ T^*_y (E) = \prod_{i=1}^{\rk E} Q_y (a_i),
\]
where $y$ is an indeterminate, coming from the power series
\[ Q_y (a) = \frac{a (1+y)}{1-e^{-a (1+y)}} - a y \in \rat [y][[a]]. \]
If $R$ is an integral domain over $\rat$, then a power series
$Q(a)\in R[[a]]$ is called \emph{normalized} if it starts with $1$, i.e.
$Q(0)=1$.
With $R=\rat [y],$ we have $Q_y (0)=1$, so $Q_y (a)$ is normalized.
For $y=0,$
\begin{equation} \label{equ.t0istd}
T^*_0 (E) = \prod_{i=1}^{\rk E} \frac{a_i}{1-e^{-a_i}} = \td^* (E)
\end{equation}
is the classical Todd class of $E$, while
for $y=1,$
\begin{equation} \label{equ.todd1ishirzelcohom}
T^*_1 (E) = \prod_{i=1}^{\rk E} \frac{a_i}{\tanh a_i} = L^* (E)
\end{equation}
is the Hirzebruch $L$-class of the vector bundle $E$.
We shall also need a certain unnormalized version of $Q_y (a)$:
Let
\[ \widetilde{Q}_y (a) = \frac{a (1+ye^{-a})}{1-e^{-a}} \in \rat [y][[a]] \]
and set
\[ \widetilde{T}^*_y (E) = \prod_{i=1}^{\rk E} \widetilde{Q}_y (a_i). \]
Note that $\widetilde{Q}_y (0) = 1+y \not= 1$, whence
$\widetilde{Q}_y (a)$ is unnormalized.
The relation
\[ (1+y) Q_y (a) = \widetilde{Q}_y ((1+y)a) \]
implies:
\begin{prop} \label{prop.relunnormtynormty}
If $E$ is a complex vector bundle of complex rank $r$, then for the degree $2i$ components:
\[ \widetilde{T}^i_y (E) = (1+y)^{r-i} T^i_y (E). \]
\end{prop}
More conceptually, we have the following formula for the
unnormalized class:
\begin{prop} \label{prop.unnormtytdtimeschlambda}
For any complex vector bundle $E$, we have
\[ \widetilde{T}^*_y (E) = \td^* (E) \cup \ch^* (\lambda_y (E^*)). \]
\end{prop}
\begin{proof}
The Chern character is given by
\[ \ch^* (\lambda_y (E^*)) =
   \prod_{i=1}^{\rk E} (1+ye^{-a_i}), \]
with $a_i$ the Chern roots of $E$,
see \cite[p. 11]{hbj}.
Thus
\begin{align*}
\widetilde{T}^*_y (E)
&= \prod_i \frac{a_i (1+ye^{-a_i})}{1-e^{-a_i}}
   = \prod_i \frac{a_i}{1-e^{-a_i}}
     \prod_i (1+ye^{-a_i}) \\
&= \td^* (E) \cup \ch^* (\lambda_y (E^*)).
\end{align*}
\end{proof}
Let $\tau_*: K_0 (X) \longrightarrow H^\BM_{2*} (X)\otimes \rat$
denote the Todd class transformation of Baum, Fulton, MacPherson.
We review, to some extent, construction and properties of this transformation.
Let
\[ \alpha^*: K^0_{\alg} (X) \longrightarrow K^0_{\operatorname{top}} (X) \]
be the forget map which takes an algebraic vector bundle to its
underlying topological vector bundle.
Composing with the Chern character, one obtains a transformation
\[ \tau^* = \ch^* \circ \alpha^*: K^0_{\alg} (X) \longrightarrow H^{2*} (X;\rat), \]
see \cite[p. 180]{bfm2}.
Baum, Fulton and MacPherson construct with the use of Bott periodicity a corresponding
homological version
\[ \alpha_*: K_0^{\alg} (X) \longrightarrow K_0^{\operatorname{top}} (X) \]
for quasi-projective varieties $X$.
Composing with the homological Chern character
\[ \ch_*: K^{\operatorname{top}}_0 (X) \longrightarrow
  H^\BM_{2*} (X;\rat), \]
where $H^\BM_*$ denotes Borel-Moore homology,
they obtain a transformation
\[ \tau_* = \ch_* \circ \alpha_*: K_0^{\alg} (X) \longrightarrow H^\BM_{2*} (X;\rat). \]
An algebraic version of this transformation is in fact available for any
algebraic scheme over a field and generalizes the
Grothendieck Riemann-Roch theorem to singular varieties.

\begin{remark} \label{rem.bfmtoddtochow}
Let $A_* (V)$ denote Chow homology of a variety $V$, i.e.
algebraic cycles in $V$ modulo rational equivalence.
Then there is a transformation
\[ \tau_*: K_0^{\alg} (X) \longrightarrow A_* (X)\otimes \rat \]
such that for a complex algebraic variety $X$, the diagram
\[ \xymatrix{
K_0^{\alg} (X) \ar[rd]^{\tau_*} \ar[d]_{\tau_*} & \\
A_* (X)\otimes \rat \ar[r]_{\cl} & H^\BM_{2*} (X;\rat)
} \]
commutes, where $\cl$ is the cycle map; see the first commutative
diagram on p. 106 of \cite[(0.8)]{bfm}.
The construction of $\tau_*$ to Chow homology is described in
Fulton's book \cite[p. 349]{fultonintth}.
Thus Todd classes are
algebraic cycles with rational coefficients that are well-defined up to rational equivalence.
\end{remark}

According to \cite[Theorem, p. 180]{bfm2}, $\tau_*$ and $\tau^*$ are compatible
with respect to cap products, i.e. the diagram
\[ \xymatrix@C=50pt{
K^0 (X) \otimes K_0 (X) \ar[r]^{\tau^* \otimes \tau_*} \ar[d]_\cap
& H^* (X;\rat) \otimes H^\BM_* (X;\rat) \ar[d]^\cap \\
K_0 (X) \ar[r]^{\tau_*} & H^\BM_* (X;\rat)
} \]
commutes. Thus, if $E$ is a vector bundle and $\Fa$ a coherent sheaf on $X$, then
\begin{equation} \label{equ.bfmtauofcap}
\tau_* ([E] \cap [\Fa]) = \ch^* (E) \cap \tau_* [\Fa].
\end{equation}
For smooth $X$,
\[ \tau_* [\oO_X] = \td^* (TX)\cap [X] = T^*_0 (TX)\cap [X]. \]
So if $E$ is a vector bundle on a smooth variety, then
\begin{equation} \label{taubfmecaposmoothx}
\tau_* ([E] \cap [\oO_X]) = (\ch^* (E) \cup \td^* (TX))\cap [X].
\end{equation}

For locally complete intersection morphisms $f:X\to Y$,
Gysin maps
\[ f^*_\BM: H^\BM_* (Y) \longrightarrow H^\BM_{*-2d} (X) \]
have been defined by Verdier \cite[\S 10]{verdierintcompl},
and Baum, Fulton and MacPherson \cite[Ch. IV, \S 4]{bfm}, where
$d$ denotes the (complex) virtual codimension of $f$.
Thus for a regular closed embedding $g$, there is a Gysin
map $g^*_\BM$ on Borel-Moore homology, which we shall also
write as $g^{!}_{\alg}$, and for a smooth morphism
$f$ of relative dimension $r$, there is a smooth pullback
$f^*_\BM: H^\BM_* (Y) \to H^\BM_{*+2r} (X)$.
Baum, Fulton and MacPherson conjectured and Verdier showed:
\begin{prop} \label{prop.bfmvrrsmpullb}
(Verdier-Riemann-Roch for smooth pullbacks.)
For a smooth morphism $f:X\to Y$ of complex algebraic varieties
and $[\Fa] \in K^{\alg}_0 (Y)$,
\[ \td^* (T_{X/Y}) \cap f^*_\BM \tau_* [\Fa] = \tau_* (f^*_K [\Fa]). \]
\end{prop}

Yokura \cite{yok} twisted $\tau_*$ by a Hirzebruch-type variable $y$:
\begin{defn}\label{definition twisted todd transformation}
The \emph{twisted Todd transformation}
\[ \td_{1+y}: K_0 (X) \otimes \intg [y^{\pm 1}] \longrightarrow
                H^\BM_{2*} (X)\otimes \rat [y^{\pm 1}, (1+y)^{-1}] \]
is given by
\[ \td_{1+y} [\Fa] := \sum_{k\geq 0} \tau_k [\Fa]\cdot \frac{1}{(1+y)^k}, \]
where the Baum-Fulton-MacPherson transformation
$\tau_*$ is extended linearly over $\intg [y^{\pm 1}]$, and
$\tau_k$ denotes the degree $2k$-component of $\tau_*$.
\end{defn}
\begin{remark} \label{rem.twistedtoddtochow}
Regarding the transformation $\tau_*$ as taking values in Chow
groups $A_* (-)\otimes \rat$ (cf. Remark \ref{rem.bfmtoddtochow}),
the above definition yields
a twisted Todd transformation
\[ \td_{1+y}: K_0 (X) \otimes \intg [y^{\pm 1}] \longrightarrow
                A_* (X)\otimes \rat [y^{\pm 1}, (1+y)^{-1}], \]
which commutes with the Borel-Moore twisted Todd transformation under the
cycle map.
\end{remark}
The definition of the motivic Hirzebruch class transformation below is due to
Brasselet, Sch\"{u}rmann and Yokura \cite{bsy}, see also Sch\"urmann's expository paper
\cite{schuermannmsri}.
\begin{defn}\label{motivic hirzebruch class transformation}
The \emph{motivic Hirzebruch class transformation} is
\[ MHT_{y*} := \td_{1+y} \circ MHC_y:
  K_0 (MHM(X)) \longrightarrow H^\BM_{2*} (X) \otimes
    \rat [y^{\pm 1}, (1+y)^{-1}]. \]
\end{defn}
For the intersection Hodge module $IC^H_X$ on a complex
purely $n$-dimensional variety $X$, we use the convention
\[ IC^H_X := j_{!*} (\rat^H_U [n]), \]
which agrees with \cite[p. 444]{schuermannmsri} and
\cite[p. 345]{ps}.
Here, $U \subset X$ is smooth, of pure dimension $n$,
Zariski-open and dense, and $j_{!*}$ denotes the
intermediate extension of mixed Hodge modules associated
to the open inclusion $j:U \hookrightarrow X$.
The underlying perverse sheaf is
$\operatorname{rat} (IC^H_X) = IC_X,$ the intersection chain sheaf, where
$\operatorname{rat}: MHM (X) \to \Per (X) = \Per (X;\rat)$ is the faithful and exact
functor that sends a mixed Hodge module to its underlying perverse sheaf.
Here, $\Per (X)$ denotes perverse sheaves on $X$ which are constructible
with respect to \emph{some} algebraic stratification of $X$.
This functor extends to a functor
$\operatorname{rat}: D^b MHM (X) \to D^b_c (X) = D^b_c (X;\rat)$ between bounded derived
categories. For every object of $D^b_c (X)$ there exists \emph{some}
algebraic stratification with respect to which the object is constructible, and these
stratifications will generally vary with the object.
Recall that a functor $F$ is \emph{conservative}, if
for every morphism $\phi$ such that $F(\phi)$ is an isomorphism,
$\phi$ is already an isomorphism.
Faithful functors on balanced categories (such as abelian or
triangulated categories) are conservative.
According to \cite[p. 218, Remark (i)]{saito90ext},
$\operatorname{rat}: D^b MHM (X) \to D^b_c (X)$ is not faithful.
We recall Lemma 5.10 from \cite{banagllgysin}:
\begin{lemma} \label{lem.derratconservative}
The functor
$\operatorname{rat}: D^b MHM (X) \to D^b_c (X)$ is conservative.
\end{lemma}
Using cones, this lemma also appears embedded in the proof of
\cite[Lemma 5.3, p. 1752]{cmss1},
see also Exercise 11.2.1 in
Maxim's book \cite{max}.
The module $IC^H_X$ is the unique simple object in the
category $MHM (X)$ which restricts to $\rat_U [n]$ over $U$.
As $U$ is smooth and pure $n$-dimensional,
$\rat^H_U [n]$ is pure of weight $n$.
Since the intermediate extension $j_{!*}$ preserves weights,
$IC^H_X$ is pure of weight $n$.
There is a duality isomorphism (polarization)
$\mathbb{D}^H_X IC^H_X \cong IC^H_X (n).$
Taking $\operatorname{rat}$, this isomorphism induces a self-duality
isomorphism
\[  \mathbb{D}_X IC_X = \mathbb{D}_X \operatorname{rat} IC^H_X \cong
\operatorname{rat} \mathbb{D}^H_X IC^H_X \cong \operatorname{rat} IC^H_X (n) \cong IC_X, \]
if an isomorphism $\rat_U (n) \cong \rat_U$ is chosen.
\begin{defn} \label{intersection generalized todd class}
(\cite{bsy}.)
The \emph{intersection generalized Todd class}
(or \emph{intersection Hirzebruch characteristic class}) is
\[ IT_{y*} (X) := MHT_{y*} [IC^H_X [-n]]
    \in H^\BM_{2*} (X) \otimes \rat [y^{\pm 1}, (1+y)^{-1}]. \]
\end{defn}
Taking up expectations formulated by Cappell and Shaneson in \cite{cs2},
this class started to be investigated in detail by Cappell, Maxim and Shaneson in
\cite{cms}.

\begin{remark} \label{rem.itisalgebraic}
The intersection characteristic class $IT_{y*} (X)$ is represented by
an algebraic cycle by Remark \ref{rem.twistedtoddtochow}.
\end{remark}

\begin{remark}
Later on in the present paper, we will be interested in the specializations $IT_{1,\ast}(X)$ for $y=1$, $IT_{0,\ast}(X)$ for $y=0$ and $IT_{-1,\ast}(X)$ for $y=-1$ in \Cref{intersection generalized todd class}.
Here, we note that specialization $y = 0$ is in fact possible because we actually have $IT_{y \ast} (X) \in H^\BM_{2 \ast} (X) \otimes \rat [y, (1+y)^{-1}]$ by \cite[p. 451, Example 5.2]{schuermannmsri}.
Specialization $y = -1$ is possible in \Cref{intersection generalized todd class} because we actually have $IT_{y \ast} (X) \in H^\BM_{2 \ast} (X) \otimes \rat [y^{\pm 1}]$ as shown in \cite[p. 465, Proposition 5.21]{schuermannmsri}.
\end{remark}

\begin{remark}\label{remark motivic chern class trafo}
Let $K_{0}(\operatorname{var}/X)$ denote the motivic relative Grothendieck group of complex algebraic varieties over $X$.
By definition, $K_{0}(\operatorname{var}/X)$ is the free abelian group generated by isomorphism classes $[f] = [f \colon Z \rightarrow X]$ of morphisms $f$ to $X$ modulo the additivity relation (or scissor relation)
$$
[f] = [f \circ i] + [f \circ j]
$$
for a closed inclusion $i \colon Y \rightarrow X$ with complement $j \colon U = Z \setminus Y \rightarrow X$.
As pointed out in \cite[Remark 5.5]{schuermannmsri}, composition of the transformations $MHC_{y}$ and $MHT_{y}$ with the natural group homomorphism
$$
\chi_{Hdg} \colon K_{0}(\operatorname{var}/X) \rightarrow K_{0}(MHM(X)), \qquad [f \colon Z \rightarrow X] \mapsto [f_{!}\mathbb{Q}_{Z}^{H}],
$$
(see \cite[Corollary 4.10]{schuermannmsri}) yields similar transformations
$$
mC_{y} := MHC_{y} \circ \chi_{Hdg}, \qquad T_{y \ast} := MHT_{y} \circ \chi_{Hdg},
$$
defined on $K_{0}(\operatorname{var}/X)$, which are studied in \cite{bsy}.
\end{remark}

The intersection Hodge modules are multiplicative under the external product
$$
\boxtimes \colon MHM(X) \times MHM(Y) \rightarrow MHM(X \times Y)
$$
as follows (compare also \cite[p. 471]{mss}).

\begin{prop}\label{prop multiplicativity of intersection Hodge module}
Let $X$ and $Y$ be pure-dimensional complex algebraic varieties.
Then there is a natural isomorphism $IC^H_{X \times Y} = IC^H_X \boxtimes IC^H_Y$ in $MHM(X \times Y)$.
\end{prop}

\begin{proof}
Let $Z$ be a pure $n$-dimensional complex algebraic variety.
The perverse sheaf $IC_{Z}(\mathbb{Q})$ is a simple object in the category of perverse sheaves (see \cite[p. 112, Th\'{e}or\`{e}me 4.3.1]{bbd} or \cite[p. 142, Corollary 8.4.13]{max}).
Since the functor $\operatorname{rat} \colon MHM(Z) \rightarrow \operatorname{Per}(Z)$ is faithful, the module $IC^H_Z$ is the unique simple object in the category $MHM(Z)$ which restricts to $\rat_U^{H} [n]$ over a dense open smooth subset $U$ of $Z$ (see \cite[p. 345, Section 14.1.4]{ps}).

Let $a$ and $b$ denote the dimensions of $X$ and $Y$, respectively.
We shall show that $IC^H_X \boxtimes IC^H_Y$ is simple in $MHM(X \times Y)$ and restricts to $\mathbb{Q}^{H}_{U \times V}[a+b]$, where $U \subset X$ and $V \subset Y$ are dense open smooth subsets.
As for the former claim, let $i \colon A \hookrightarrow IC^H_X \boxtimes IC^H_Y$ be a subobject.
As $\operatorname{rat}$ is exact, its application to $i$ yields a subobject $\operatorname{rat}(i) \colon \operatorname{rat}(A) \hookrightarrow \operatorname{rat}(IC^H_X \boxtimes IC^H_Y)$.
As $\operatorname{rat}$ preserves external products (see \cite[Eq. (1.5.2)]{saitoformalism}), there is a natural isomorphism
$$
\operatorname{rat}(IC^H_X \boxtimes IC^H_Y) = \operatorname{rat}(IC^H_X) \boxtimes \operatorname{rat}(IC^H_Y) = IC_{X} \boxtimes IC_{Y}.
$$
In the category of perverse sheaves, we have $IC_{X} \boxtimes IC_{Y} = IC_{X \times Y}$, as a standard verification of the intersection sheaf axioms shows (see \cite[p. 121, Section 6.3]{gmih2}).
Therefore, we obtain a subobject $\operatorname{rat}(i) \colon \operatorname{rat}(A) \hookrightarrow IC_{X \times Y}$ in $\operatorname{Per}(X \times Y)$.
Since $IC_{X \times Y}$ is a simple object in $\operatorname{Per}(X \times Y)$, we obtain either that $\operatorname{rat}(A) = 0$ or that $\operatorname{rat}(i)$ is an isomorphism.
Faithful functors on balanced categories (such as abelian or triangulated categories) are conservative.
Thus, if $\operatorname{rat}(i)$ is an isomorphism, then $i$ is an isomorphism as well.
On the other hand, let us suppose that $\operatorname{rat}(A) = 0$.
Application of the additive functor $\operatorname{rat}$ to the unique morphism $t \colon A \rightarrow 0$ yields $\operatorname{rat}(t) \colon \operatorname{rat}(A) \rightarrow \operatorname{rat}(0) = 0$, which is an isomorphism since $\operatorname{rat}(A) = 0$.
Using again that $\operatorname{rat}$ is conservative, we conclude that $A = 0$.
This shows that $IC^H_X \boxtimes IC^H_Y$ is simple in $MHM(X \times Y)$.
Next, let us show that there is a natural isomorphism
\begin{equation}\label{equation claim restriction multiplicativity}
(i_{U} \times i_{V})^{\ast}(IC^H_X \boxtimes IC^H_Y) = \mathbb{Q}^{H}_{U \times V}[a+b] \in MHM(U \times V),
\end{equation}
where $i_{U} \colon U \hookrightarrow X$ and $i_{V} \colon V \hookrightarrow Y$ are inclusions of dense open smooth subsets.
We will first establish the isomorphism in $D^{b}MHM(U \times V)$, and then show that it gives rise to one in $MHM(U \times V)$.
For $M \in D^{b}MHM(X)$ and $N \in D^{b}MHM(Y)$, there is a natural isomorphism
$$
(i_{U} \times i_{V})^{\ast}(M \boxtimes N) = i_{U}^{\ast} M \boxtimes i_{V}^{\ast}N \in D^{b}MHM(U \times V)
$$
by \cite[Eq. (3.1.9), p. 13]{saitoformalism}.
Substituting $M = IC^H_X$ and $N = IC^H_Y$, we obtain a natural isomorphism
$$
(i_{U} \times i_{V})^{\ast}(IC^H_X \boxtimes IC^H_Y) = i_{U}^{\ast} IC^H_X \boxtimes i_{V}^{\ast}IC^H_Y \in D^{b}MHM(U \times V).
$$
Furthermore, it follows from $i_{U}^{\ast} IC^H_X = \mathbb{Q}^{H}_{U}[a]$ and $i_{V}^{\ast} IC^H_Y = \mathbb{Q}^{H}_{V}[b]$ that
$$
i_{U}^{\ast} IC^H_X \boxtimes i_{V}^{\ast}IC^H_Y = \mathbb{Q}^{H}_{U}[a] \boxtimes \mathbb{Q}^{H}_{V}[b] = (\mathbb{Q}^{H}_{U} \boxtimes \mathbb{Q}^{H}_{V})[a+b] \in D^{b}MHM(U \times V).
$$
For the projection $p \colon U \times V \rightarrow V$ to the second factor, there are natural isomorphisms
$$
\mathbb{Q}^{H}_{U} \boxtimes \mathbb{Q}^{H}_{V} = p^{\ast}\mathbb{Q}^{H}_{V} = \mathbb{Q}^{H}_{U \times V} \in D^{b}MHM(U \times V)
$$
by \cite[Eq. (3.6.3), p. 15]{saitoformalism} and \cite[Eq. (3.9.1), p. 16]{saitoformalism}, respectively.
By composition, we arrive at a natural isomorphism
\begin{equation}\label{equation claim restriction multiplicativity in derived category}
(i_{U} \times i_{V})^{\ast}(IC^H_X \boxtimes IC^H_Y) = \mathbb{Q}^{H}_{U \times V}[a+b] \in D^{b}MHM(U \times V).
\end{equation}
Since $MHM(-)$ is stable under external product and pullback by open embeddings, it follows from $IC^H_X \in MHM(X)$ and $IC^H_Y \in MHM(Y)$ that $(i_{U} \times i_{V})^{\ast}(IC^H_X \boxtimes IC^H_Y) \in MHM(U \times V)$.
Moreover, we have $\mathbb{Q}^{H}_{U \times V}[a+b] \in MHM(X \times Y)$ because $U \times V$ is smooth.
Finally, application of the functor $H^{0} \colon D^{b}MHM(U \times V) \rightarrow MHM(U \times V)$ to the isomorphism (\ref{equation claim restriction multiplicativity in derived category}) yields the desired isomorphism (\ref{equation claim restriction multiplicativity}).
\end{proof}

\begin{cor}\label{corollary multiplicativity of intersection generalized todd class}
The intersection generalized Todd class of the product $X \times Y$ of complex algebraic pure-dimensional varieties $X, Y$ satisfies
$$
IT_{y*} (X \times Y) = IT_{y*} (X) \times IT_{y*} (Y) \in H^\BM_{2*} (X \times Y) \otimes \rat [y^{\pm 1}, (1+y)^{-1}].
$$
\end{cor}

\begin{proof}
By \cite[Corollary 5.10, p. 458]{schuermannmsri}, the motivic Chern class transformation
\[ MHC_y: K_0 (MHM(X)) \to K^{\alg}_0 (X) \otimes \intg [y^{\pm 1}] \]
(see \Cref{motivic Hodge Chern class transformation}) commutes with the external product, that is,
$$
MHC_y([M] \boxtimes [N]) = MHC_y([M]) \boxtimes MHC_y([N])
$$
for $M \in D^{b}MHM(X)$ and $N \in D^{b}MHM(Y)$.
Furthermore, according to \cite[p. 462]{schuermannmsri}, the Todd class transformation
$$
\tau_{\ast} \colon K_{0}^{\alg} (X) \longrightarrow H^\BM_{2*} (X)\otimes \rat
$$
of Baum, Fulton, MacPherson is compatible with the external product in the sense that
$$
\tau_{\ast}([\mathcal{F}] \boxtimes [\mathcal{G}]) = \tau_{\ast}([\mathcal{F}]) \times \tau_{\ast}([\mathcal{G}])
$$
for $[\mathcal{F}] \in K_{0}(X)$ and $[\mathcal{G}] \in K_{0}(Y)$ (see also \cite[Example 18.3.1, p. 360, and Example 19.1.9, p. 377]{fultonintth}, as well as \cite[Property (i), p. 122]{fmp}).
(The cross product on Borel-Moore homology agrees with the cross product on ordinary singular homology under the isomorphism $H_{\ast}^{\BM}(X) \cong H_{\ast}(\overline{X}, \overline{X} - X)$, where $\overline{X}$ is a compactification of $X$ such that $(\overline{X}, \overline{X} - X)$ is a CW pair, see \cite[p. 99, 2.6.19]{chrissginzburg}.)
The latter implies for the twisted Todd transformation
\[
\td_{1+y}: K_{0}^{\alg} (X) \otimes \intg [y^{\pm 1}] \longrightarrow
                H^\BM_{2*} (X)\otimes \rat [y^{\pm 1}, (1+y)^{-1}]
\]
of \Cref{definition twisted todd transformation} that
\begin{align*}
td_{1+y}([\mathcal{F}]) \times td_{1+y}([\mathcal{G}]) &= \left[\sum_{k \geq 0}\tau_{k}([\mathcal{F}]) \frac{1}{(1+y)^k}\right] \times \left[\sum_{l \geq 0}\tau_{l}([\mathcal{G}]) \frac{1}{(1+y)^l}\right] \\
&= \sum_{k, l \geq 0} \tau_{k}([\mathcal{F}]) \times \tau_{l}([\mathcal{G}]) \frac{1}{(1+y)^{k+l}} \\
&= \sum_{m \geq 0} \left[\sum_{k+l = m}\tau_{k}([\mathcal{F}]) \times \tau_{l}([\mathcal{G}])\right] \frac{1}{(1+y)^{m}} \\
&= \sum_{m \geq 0} \tau_{m}([\mathcal{F}] \boxtimes [\mathcal{G}]) \frac{1}{(1+y)^{m}} \\
&= td_{1+y}([\mathcal{F}] \boxtimes [\mathcal{G}])
\end{align*}
for $[\mathcal{F}] \in K_0 (X) \otimes \intg [y^{\pm 1}]$ and $[\mathcal{G}] \in K_0 (Y) \otimes \intg [y^{\pm 1}]$.
By composition, we see that the motivic Hirzebruch class transformation
\[ MHT_{y*} := \td_{1+y} \circ MHC_y:
  K_0 (MHM(X)) \longrightarrow H^\BM_{2*} (X) \otimes
    \rat [y^{\pm 1}, (1+y)^{-1}] \]
(see \Cref{motivic hirzebruch class transformation}) satisfies
$$
MHT_{y \ast}([M] \boxtimes [N]) = MHT_y([M]) \times MHT_y([N])
$$
for $M \in D^{b}MHM(X)$ and $N \in D^{b}MHM(Y)$.
Finally, by taking $M = IC^H_X [-a]$ and $M = IC^H_Y [-b]$, where $a$ and $b$ denote the dimensions of $X$ and $Y$, respectively, we can compute the intersection generalized Todd class (see \Cref{intersection generalized todd class}) of the product $X \times Y$ to be
\begin{align*}
IT_{y*} (X \times Y) &= MHT_{y \ast}[IC^H_{X \times Y} [-(a+b)]] \\
&= MHT_{y \ast}[IC^H_{X}[-a] \boxtimes IC^H_{Y} [-b]] \\
&= MHT_{y \ast}([IC^H_{X}[-a]] \boxtimes [IC^H_{Y} [-b]]) \\
&= MHT_{y \ast}[IC^H_{X}[-a]] \times MHT_{y \ast}[IC^H_{Y} [-b]] \\
&= IT_{y*} (X) \times IT_{y*} (Y),
\end{align*}
where we exploited that $IC^H_{X \times Y} = IC^H_X \boxtimes IC^H_Y \in MHM(X \times Y)$ by \Cref{prop multiplicativity of intersection Hodge module}.
\end{proof}

\begin{prop}\label{prop isomorphism invariance of of intersection generalized todd class}
Let $\Phi \colon X \stackrel{\cong}{\longrightarrow} Y$ be an isomorphism of pure-dimensional complex algebraic varieties.
Then, $\Phi_{\ast}IT_{y*} (X) = IT_{y*} (Y)$ in $H^\BM_{2*} (Y) \otimes \rat [y^{\pm 1}, (1+y)^{-1}]$.
\end{prop}

\begin{proof}
This follows from $\Phi_{\ast}IC^H_{X} = IC^H_{Y}$, and by naturality of $MHT_{y \ast}$.
\end{proof}

\section{Blow-up and Transversality for Complex Manifolds}\label{Blow-up and transversality}
This section studies stratifications and transversality of the strict transform of Whitney stratified sets under blow-ups along smooth centers.
Throughout, let $M$ be a complex submanifold of a complex manifold $W$.

\begin{lemma}\label{lemma complex transversality and normal bundles}
Let $S \subset W$ be a complex submanifold that is transverse to $M \subset W$ (in the sense of $C^{\infty}$ manifolds).
Then, the intersection $S \cap M$ is a complex submanifold of $W$.
Moreover, the canonical map $N_{S\cap M}S \rightarrow (N_{M}W)|_{S\cap M}$ of complex normal bundles is an isomorphism.
\end{lemma}

\begin{proof}
The proof of the first claim is similar to the real smooth version, but is based on the complex counterpart of the $C^{\infty}$ implicit function theorem (see e.g. Grifiths-Harris \cite[p. 19]{gh}).
The proof of the second claim uses the definition of the complex normal bundle $N_{M}W$ as the quotient of the restricted complex tangent bundle $TW|_{M}$ by the subbundle $TM$ (see e.g. Griffiths-Harris \cite[p. 71]{gh}).
\end{proof}

In the following, let $\pi: W' = \operatorname{Bl}_{M}W \rightarrow W$ denote the blow-up of $W$ along $M$.
Thus, $W'$ is a complex manifold, $\pi$ is a holomorphic map that is an isomorphism away from $M$, $\pi| \colon W' \setminus \pi^{-1}(M) \stackrel{\cong}{\longrightarrow} W \setminus M$, and the restriction $\pi| \colon E \rightarrow M$ of $\pi$ to the exceptional divisor $E := \pi^{-1}(M) \subset W'$ is isomorphic over $M$ to the projectivization $\mathbb{P}(N_{M}W) \rightarrow M$ of the complex normal bundle $N_{M}W \rightarrow M$ of $M$ in $W$.

The blow-up of $W$ along $M$ can be characterized uniquely up to isomorphism as follows.
Given a complex manifold $W'_{0}$ and a holomorphic map $\pi_{0} \colon W'_{0} \rightarrow W$ that is an isomorphism away from $M$, and such that the fibers of $\pi_{0}$ over points of $M$ are isomorphic to complex projective $(k-1)$-space $\mathbb{P}(\mathbb{C}^{k})$, where $k$ denotes the codimension of $M$ in $W$, there is an isomorphism $\Phi \colon W_{0}' \stackrel{\cong}{\longrightarrow} W'$ such that $\pi_{0} = \pi \circ \Phi$.
Consequently, we have the following

\begin{prop}\label{lemma locality of blow up}
For every open subset $U \subset W$, the restriction $\pi| \colon \pi^{-1}(U) \rightarrow U$ is the blow-up of $U$ along the complex submanifold $M \cap U \subset U$.
\end{prop}

The inverse image $\pi^{-1}(Z)$ of a subset $Z \subset W$ is called the \emph{total transform} of $Z$ under $\pi$.
For any subset $Z \subset W$, the \emph{strict transform} $\tilde{Z}$ of $Z$ under $\pi$ is the intersection of the total transform $\pi^{-1}(Z)$ with the closure of the inverse image $\pi^{-1}(Z \setminus M)$ in $W'$.
In particular, for a closed subset $Z \subset W$, we note that $\tilde{Z}$ is just the closure of the inverse image $\pi^{-1}(Z \setminus M)$ in $W'$.
In view of \Cref{lemma locality of blow up}, it makes sense to study the behavior of the strict transform under restriction to open subsets.

\begin{lemma}\label{lemma locality of strict transform}
For any subset $Z \subset W$, the strict transform of the intersection $Z \cap U$ under the blow-up $\pi| \colon \pi^{-1}(U) \rightarrow U$ of an open subset $U \subset W$ along $M \cap U$ is given by $\tilde{Z} \cap \pi^{-1}(U)$.
\end{lemma}

\begin{proof}
The strict transform of the intersection $Z \cap U$ under the blow-up $\pi| \colon \pi^{-1}(U) \rightarrow U$ of an open subset $U \subset W$ along $M \cap U$ is by definition the intersection of the total transform $\pi^{-1}(Z \cap U) = \pi^{-1}(Z) \cap \pi^{-1}(U)$ with the closure of the inverse image $\pi^{-1}((Z \cap U) \setminus (M \cap U)) = \pi^{-1}(Z \setminus M) \cap \pi^{-1}(U)$ in $\pi^{-1}(U)$.
Since $\pi^{-1}(U) \subset W'$ is an open subset, the closure of $\pi^{-1}(Z \setminus M) \cap \pi^{-1}(U)$ in $\pi^{-1}(U)$ equals the intersection of $\pi^{-1}(U)$ with the closure of $\pi^{-1}(Z \setminus M)$ in $W'$.
All in all, we obtain the intersection of $\pi^{-1}(Z) \cap \pi^{-1}(U)$ with the closure of $\pi^{-1}(Z \setminus M)$ in $W'$, which is just $\tilde{Z} \cap \pi^{-1}(U)$.
\end{proof}

\begin{prop}\label{proposition strict and total transform agree for strata}
If $S \subset W$ is a complex submanifold that is transverse to $M \subset W$, then the strict transform and the total transform of $S$ under $\pi$ coincide, $\tilde{S} = \pi^{-1}(S)$.
\end{prop}

\begin{proof}
Without loss of generality, we may assume that $S$ is a closed subset of $W$.
(In fact, $S$ is a closed subset of an open tubular neighborhood $U \subset W$ of $S$, and the restriction $\pi| \colon \pi^{-1}(U) \rightarrow U$ is the blow-up of $U$ along the complex submanifold $M \cap U \subset U$ by \Cref{lemma locality of blow up}.
Hence, the strict transform of $S \cap U = S$ under $\pi| \colon \pi^{-1}(U) \rightarrow U$ coincides with the total transform $\pi^{-1}(S)$.
On the other hand, this strict transform is given by $\tilde{S} \cap \pi^{-1}(U) = \tilde{S}$ according to \Cref{lemma locality of strict transform}.)
Then, according to Griffiths-Harris \cite[property 5, pp. 604-605]{gh}, the intersection $\tilde{S} \cap E$ corresponds under the isomorphism $\Psi \colon E \cong \mathbb{P}(N_{M}W)$ of bundles over $M$ to the image in $N_{M}W$ of the complex tangent spaces $T_{x}S$ to the points $x \in S \cap M$.
This image is $(N_MW)|_{S\cap M}$ because the right-hand vertical arrow in the commutative diagram
\[ \xymatrix{
TS|_{S \cap M} \ar[r]^{\operatorname{quot}} \ar[d]_{\operatorname{incl}} & N_{S \cap M}S \ar[d]^{\cong} \\
TW|_{S \cap M} \ar[r]_{\operatorname{quot}} & (N_{M}W)|_{S \cap M}
} \]
is surjective by \Cref{lemma complex transversality and normal bundles}.
Consequently,
$$
\tilde{S} \cap E = \Psi^{-1}\mathbb{P}((N_MW)|_{S\cap M}) = \pi^{-1}(S \cap M).
$$
All in all, we obtain
$$
\tilde{S} = (\tilde{S} \cap E) \cup (\tilde{S} \cap \pi^{-1}(S \setminus M)) = \pi^{-1}(S \cap M) \cup \pi^{-1}(S \setminus M) = \pi^{-1}(S).
$$
\end{proof}

From now on, we assume familiarity with the notion of Whitney stratifications; for details, see e.g. \cite[Section 1.2, p. 37]{gmsmt}.

\begin{cor}\label{corollary strict and total transform agree}
If $X \subset W$ is a Whitney stratified subspace whose strata are complex submanifolds of $W$ that are transverse to $M \subset W$, then the strict transform and the total transform of $X$ coincide, $\tilde{X} = \pi^{-1}(X)$.
\end{cor}

\begin{proof}
Since every stratum $S$ of $X$ satisfies $\pi^{-1}(S) = \tilde{S}$ by \Cref{proposition strict and total transform agree for strata}, and $\tilde{S} \subset \tilde{X}$ holds by definition of the strict transform, we have
$$
\pi^{-1}(X) = \bigcup_{S} \pi^{-1}(S) = \bigcup_{S}\tilde{S} \subset \tilde{X}.
$$
Conversely, we have $\tilde{X} \subset \pi^{-1}(X)$ by definition of the strict transform.
\end{proof}

We proceed to the main result of this section.

\begin{thm}\label{main result blowup}
Let $X \subset W$ be a Whitney stratified subspace whose strata are complex submanifolds of $W$ that are transverse to $M \subset W$.
Then, the strict transform $\tilde{X} \subset W'$ admits a Whitney stratification whose strata are transverse to the exceptional divisor $E \subset W'$ of $\pi$.
\end{thm}

In view of \Cref{proposition strict and total transform agree for strata} and \Cref{corollary strict and total transform agree}, \Cref{main result blowup} follows from the following

\begin{thm}\label{main result blowup total transform version}
Let $X \subset W$ be a Whitney stratified subspace whose strata are transverse to $M \subset W$.
Then, the total transforms $\pi^{-1}(S)$ of the strata $S$ of $X$ form a Whitney stratification of the total transform $\pi^{-1}(X) \subset W'$.
Furthermore, every stratum of the total transform $\pi^{-1}(X)$ is transverse to $E = \pi^{-1}(M) \subset W'$.
\end{thm}

\begin{proof}
The claims are local in the sense that it suffices to find an open cover of $W$ such that for every subset $U \subset W$ of this cover, the inverse images $\pi^{-1}(S \cap U)$ of the strata $S \cap U$ of $X \cap U$ form a Whitney stratification of the inverse image $\pi^{-1}(X \cap U) \subset \pi^{-1}(U)$, and every $\pi^{-1}(S \cap U)$ is transverse to $\pi^{-1}(M \cap U) \subset \pi^{-1}(U)$.
Since $\pi \colon W' \rightarrow W$ is locally isomorphic to the blow-up of an $n$-dimensional open disc along a coordinate plane (see property 3 in \cite[p. 604]{gh}), we may therefore assume without loss of generality that $\pi$ is globally of this form.
Explicitly, the blow-up $\pi: W' \rightarrow W$ of an $n$-dimensional open disc $W \subset \mathbb{C}^{n}$ with complex coordinates $z = (z_{1}, \dots, z_{n})$ along the coordinate plane $M = W \cap \{z_{i} = 0; \; i = 1, \dots, k\}$ of codimension $k$ is given by restricting the projection $W \times \mathbb{P}(\mathbb{C}^{k}) \rightarrow W$ to the first factor to
$$
W' = \{(z, w = [w_{1} : \dots : w_{k}]) \in W \times \mathbb{P}(\mathbb{C}^{k}); \; z_i w_j=z_j w_i \text{ for all } 1 \leq i, j \leq k\},
$$
and we have $E = \pi^{-1}(M) = M \times \mathbb{P}(\mathbb{C}^{k})$.
In this situation, we have the following

\begin{lemma}\label{lemma transversality in blow-up of open disk}
If $S \subset W$ is a smooth submanifold that is transverse to $M \subset W$, then $S \times \mathbb{P}(\mathbb{C}^{k})$ is transverse to $E = M \times \mathbb{P}(\mathbb{C}^{k}) \subset W \times \mathbb{P}(\mathbb{C}^{k})$ and to $W' \subset W \times \mathbb{P}(\mathbb{C}^{k})$.
\end{lemma}

\begin{proof}
Since $S$ is transverse to $M \subset W$, it follows that $S \times \mathbb{P}(\mathbb{C}^{k})$ is transverse to $E = M \times \mathbb{P}(\mathbb{C}^{k}) \subset W \times \mathbb{P}(\mathbb{C}^{k})$, which proves the first claim.
To show the second claim, it remains to prove transversality of $S$ and $W'$ at points in $(S \times \mathbb{P}(\mathbb{C}^{k})) \cap (W' \setminus E)$.
For this purpose, fix a point
$$
x = (y, z) \in (S \times \mathbb{P}(\mathbb{C}^{k})) \cap (W' \setminus E) \; (\subset (S \setminus M) \times \mathbb{P}(\mathbb{C}^{k})),
$$
and let $(p, q) \in T_{x}(W \times \mathbb{P}(\mathbb{C}^{k})) = T_{y}(W \setminus M) \times T_{z}\mathbb{P}(\mathbb{C}^{k})$ be an arbitrary tangent vector.
Since the differential
$$
d_{x}\pi: T_{x}(W'\setminus E) \to T_{y}(W\setminus M)
$$
is surjective, the vector $p \in T_{y}(W\setminus M)$ is the image under $d_{x}\pi$ of a vector in $T_{x}(W'\setminus E)$ of the form $(p, q') \in T_{x}(W'\setminus E) \subset T_{y}(W \setminus M) \times T_{z}\mathbb{P}(\mathbb{C}^{k})$.
Writing $(p, q) = (p, q') + (0, q-q')$ with $(p, q') \in T_{x}(W'\setminus E)$ and $(0, q-q') \in T_{x}(S \times \mathbb{P}(\mathbb{C}^{k})) = T_{y}S \times T_{z}\mathbb{P}(\mathbb{C}^{k})$, we conclude that $S \times \mathbb{P}(\mathbb{C}^{k})$ and $W'$ are transverse at $x$, which proves the second claim.
\end{proof}

Using \Cref{lemma transversality in blow-up of open disk}, we can now complete the proof of \Cref{main result blowup total transform version}.
For this purpose, let $X \subset W$ be a Whitney stratified subspace whose strata $S$ are transverse to $M \subset W$.
We have to show that the total transforms $\pi^{-1}(S)$ form a Whitney stratification of the total transform $\pi^{-1}(X) \subset W'$, and that every stratum $\pi^{-1}(S)$ of the total transform $\pi^{-1}(X)$ is transverse to $E \subset W'$.
Since $\pi^{-1}(Z) = W' \cap (Z \times \mathbb{P}(\mathbb{C}^{k}))$ for all subsets $Z \subset W$, we obtain the claims by applying \Cref{lemma transversality of intersections} below to the smooth submanifolds $E \subset W'$ of the smooth manifold $W \times \mathbb{P}(\mathbb{C}^{k})$, and to the Whitney stratified subset $X \times \mathbb{P}(\mathbb{C}^{k}) \subset W \times \mathbb{P}(\mathbb{C}^{k})$, whose strata $S \times \mathbb{P}(\mathbb{C}^{k})$ are all transverse to $E \subset W \times \mathbb{P}(\mathbb{C}^{k})$ and to $W' \subset W \times \mathbb{P}(\mathbb{C}^{k})$ by \Cref{lemma transversality in blow-up of open disk}.

\begin{lemma}\label{lemma transversality of intersections}
Let $N \subset M$ be smooth submanifolds of a smooth manifold $W$.
Let $X \subset W$ be a Whitney stratified subset whose strata $S$ are transverse to $N \subset W$ and to $M \subset W$.
Then, the intersection $Y := M \cap X \subset M$ is a Whitney stratified subset whose strata $M \cap S$ are transverse to $N \subset M$.
\end{lemma}

\begin{proof}
By the theory of Whitney stratifications (see e.g. Lemma 2.2.2 in \cite{cheniot}), it is well-known that the intersection $Y := M \cap X \subset M$ is a Whitney stratified subset with strata $S \cap M$, where $S$ runs through the strata of $X$.
It remains to show that the strata of $Y$ are transverse to $N \subset M$.
For this purpose, we fix $x \in N \cap Y$, and have to show that $T_{x}M \subset T_{x}N + T_{x}S'$, where $S'$ denotes the stratum of $Y = M \cap X$ that contains $x$.
Now, $N \pitchfork_{P} X$ and $x \in N \cap Y \subset N \cap X$ imply that $T_{x}P \subset T_{x}N + T_{x}S$, where $S$ denotes the stratum of $X$ that contains $x$.
Therefore, using $M \subset P$, we obtain $T_{x}M \subset T_{x}N + T_{x}S$.
Thus, given $\mu \in T_{x}M$, we find $\nu \in T_{x}N$ and $\xi \in T_{x}S$ such that $\mu = \nu + \xi$.
Using $T_{x}N \subset T_{x}M$, we conclude that
\[
\xi = \mu - \nu \in T_{x}M \cap T_{x}S = T_{x}(M \cap S) = T_{x}S'.
\]
All in all, we have $\mu = \nu + \xi \in T_{x}N + T_{x}S'$.
As $\mu \in T_{x}M$ was arbitrary, the claim follows.
\end{proof}

This completes the proof of \Cref{main result blowup total transform version}.
\end{proof}

\begin{remark}
Our proof of \Cref{main result blowup total transform version} is inspired by an argument of Cheniot \cite[pp. 145-146]{cheniot}.
There, one has $M=A$ of codimension $k = 2$, and thus $P(\mathbb{C}^k) =\mathbb{P}^1(\mathbb{C})$, where $\phi:=\pi \colon Z:=W' \to W$ denotes the blow-up.
However, all we really need is that $M$ is transverse to the Whitney stratification of $X$ in $W$.
\end{remark}

\section{Algebraic Gysin Restriction of $IT_{1, \ast}$ in a Transverse Setup}\label{review transverse setup}
In \Cref{Topological normally nonsingular embeddings}, we review from \cite{banagllgysin} a topological notion of normally nonsingular embedding that respects a stratification of the ambient space.
Then, \Cref{Algebraic Gysin restriction for normally nonsingular embeddings} discusses along the lines of \cite{banagllgysin} the first author's Verdier-Riemann-Roch type formula for the algebraic Gysin restriction of the Hodge-theoretic characteristic classes $IT_{1, \ast}$ with respect to suitably normally nonsingular closed algebraic regular embeddings.
Finally, in \Cref{Algebraic Gysin restriction in transverse setup}, we derive our algebraic Gysin restriction formula for $IT_{1, \ast}$ in a transverse setup (see \Cref{theorem gysin restriction of hodge theoretic cc}) by invoking \Cref{main result blowup}.

\subsection{Topological Normally Nonsingular Embeddings}\label{Topological normally nonsingular embeddings}
\begin{defn}[see Definition 3.1 in \cite{banagllgysin}]
A \emph{topological stratification} of a topological space $X$ is a filtration
$$
X = X_{n} \supseteq X_{n-1} \supseteq \dots \supseteq X_{1} \supseteq X_{0} \supseteq X_{-1} = \varnothing
$$
by closed subsets $X_{i}$ such that the difference sets $X_{i} - X_{i-1}$ are topological manifolds of pure dimension $i$, unless empty.
The connected components $X_{\alpha}$ of these difference sets are called the \emph{strata}.
We will often write stratifications as $\mathcal{X} = \{X_{\alpha}\}$.
\end{defn}

The following definition of locally cone-like topological stratifications, which are also known as \emph{cs-stratifications}, is due to Siebenmann \cite{siebenmann}; see also \cite[Def. 4.2.1, p. 232]{schuermannbook}.

{\begin{defn}[see Definition 3.2 in \cite{banagllgysin}]
A topological stratification $\{X_{i}\}$ of $X$ is called \emph{locally cone-like} if for all $x \in X_{i} - X_{i-1}$ there is an open neighborhood $U$ of $x$ in $X$, a compact topological space $L$ with filtration
$$
L = L_{n-i-1} \supseteq L_{n-i-2} \supseteq \dots \supseteq L_{0} \supseteq L_{-1} = \varnothing,
$$
and a filtration preserving homeomorphism $U \cong \mathbb{R}^{i} \times \operatorname{cone}^{\circ}(L)$, where $\operatorname{cone}^{\circ}(L)$ denotes the open cone on $L$.
\end{defn}

We shall employ the following notion of normally nonsingular embedding of topological spaces that respects an ambient locally cone-like topological stratification.

\begin{defn}[see Definition 3.3 in \cite{banagllgysin}] \label{definition normally nonsingular embedding}
Let $X$ be a topological space with locally cone-like topological stratification $\mathcal{X} = \{X_{\alpha}\}$ and let $Y$ be any topological space.
An embedding $g \colon Y \hookrightarrow X$ is called \emph{normally nonsingular} (with respect to $\mathcal{X}$), if
\begin{enumerate}
\item $\mathcal{Y} := \{Y_{\alpha} := X_{\alpha} \cap Y\}$ is a locally cone-like topological stratification of $Y$,
\item there exists a topological vector bundle $\pi \colon E \rightarrow Y$ and
\item there exists a (topological) embedding $j \colon E \rightarrow X$ such that
\begin{enumerate}
\item $j(E)$ is open in $X$,
\item $j|_{Y} = g$, and
\item the homeomorphism $j \colon E \stackrel{\cong}{\longrightarrow} j(E)$ is stratum preserving, where the open set $j(E)$ is endowed with the stratification $\{X_{\alpha} \cap j(E)\}$ and $E$ is endowed with the stratification $\mathcal{E} = \{\pi^{-1}Y_{\alpha}\}$.
\end{enumerate}
\end{enumerate}
\end{defn}
Note that the above stratification $\mathcal{E}$ of the total space $E$ is automatically topologically locally cone-like.
For example, transverse intersections give rise to normally nonsingular inclusions as follows (compare the beginning of the proof of Proposition 6.3 in \cite{banagllgysin}).

\begin{thm}\label{normally nonsingular inclusions induced by transverse intersections}
(See Theorem 1.11 in \cite[p. 47]{gmsmt}.)
Let $X \subset W$ be a Whitney stratified subset of a smooth manifold $W$.
Suppose that $M \subset W$ is a smooth submanifold of codimension $r$ that is transverse to every stratum of $X$, and that $Y = M \cap X$ is compact.
Then, the inclusion $g \colon Y \hookrightarrow X$ is normally nonsingular of codimension $r$ with respect to the normal bundle $\nu = \nu_{M \subset W}|_{Y}$ given by restriction of the normal bundle $\nu_{M \subset W}$ of $M$ in $W$.
\end{thm}

\begin{proof}
In \cite{pati}, Pati points out that one may wish to add precision to the argument given in \cite[p. 48]{gmsmt}, as Thom's First Isotopy Lemma is applied there to a composition $\Psi^{-1}(X) \subset E_{\varepsilon} \times (-\delta, 1+\delta) \rightarrow (-\delta, 1+\delta)$ which is not necessarily proper.
To address this, we impose the assumption that $Y = M \cap X$ be compact, and apply Thom's First Isotopy Lemma to the proper composition $\Psi^{-1}(X) \cap D_{\varepsilon/2} \subset E_{\varepsilon} \times (-\delta, 1+\delta) \rightarrow (-\delta, 1+\delta)$, where $D_{\varepsilon/2} \subset E_{\varepsilon}$ denotes the closed disk bundle of radius $\varepsilon/2$.
To assure that the intersection $\Psi^{-1}(X) \cap D_{\varepsilon/2}$ is Whitney stratified, we choose $\varepsilon > 0$ small enough to achieve, using Whitney's condition B, that $\Psi|_{\partial D_{\varepsilon/2}}$ is transverse to $X \subset W$.
\end{proof}

\subsection{Algebraic Gysin Restriction of $IT_{1, \ast}$ for Upwardly Normally Nonsingular Embeddings} \label{Algebraic Gysin restriction for normally nonsingular embeddings}
We review the Verdier-Riemann-Roch type formula for the algebraic Gysin restriction of $IT_{1, \ast}$ obtained in \cite{banagllgysin}.

\begin{defn}[see p. 1279 in \cite{banagllgysin}] \label{definition algebraic stratification}
An \emph{algebraic stratification} of a complex algebraic variety $X$ is a locally cone-like topological stratification $\{X_{2i}\}$ of $X$ such that the closed subspaces $X_{2i}$ are algebraic subsets of $X$.
\end{defn}

\begin{example}\label{example complex algebraic Whitney stratifications are algebraic stratifications}
Let $Z$ be a closed subvariety of a smooth complex algebraic variety $W$.
A Whitney stratification of $Z \subset W$ is called \emph{complex algebraic} if all its open strata are smooth complex algebraic subvarieties of $W$.
It is well-known that complex algebraic Whitney stratifications always exist (see e.g. \cite[Theorem, p. 43]{gmsmt}).
Furthermore, complex algebraic Whitney stratifications are algebraic stratifications in the sense of \Cref{definition algebraic stratification}.
(Here, we point out that the closures of the open strata are the same in the Zariski and the complex topology; see e.g. \cite[Corollary 1, p. 60]{mumfordred}.)
\end{example}

\begin{defn}[see Definition 3.4 in \cite{banagllgysin}] \label{compatibly stratifiable embeddings}
If $X$ and $Y$ are complex algebraic varieties and $g \colon Y \hookrightarrow X$ a closed algebraic embedding whose underlying topological embedding $g_{\operatorname{top}}$ in the complex topology is normally nonsingular, then we will call $g$ and $g_{\operatorname{top}}$ \emph{compatibly stratifiable} if there exists an algebraic stratification $\mathcal{X}$ of $X$ such that $g_{\operatorname{top}}$ is normally nonsingular with respect to $\mathcal{X}$ and the induced stratification $\mathcal{Y}$ is an algebraic stratification of $Y$.
\end{defn}

The algebraic normal bundle of a regular algebraic embedding does not necessarily reflect the normal topology near the subvariety.
In particular, the underlying topological embedding needs not be normally nonsingular.
This observation motivates the following

\begin{defn}[see Definition 6.1 in \cite{banagllgysin}]\label{definition tight embedding}
A closed regular algebraic embedding $Y \hookrightarrow X$ of complex algebraic varieties is called \emph{tight}, if its underlying topological embedding (in the complex topology) is normally nonsingular and compatibly stratifiable, with topological normal bundle $\pi \colon E \rightarrow Y$ as in \Cref{definition normally nonsingular embedding}, and $E \rightarrow Y$ is isomorphic (as a topological vector bundle) to the underlying topological vector bundle of the algebraic normal bundle $N_{Y}X$ of $Y$ in $X$.
\end{defn}

For example, closed embeddings of smooth complex varieties are tight, see Example 6.2 in \cite{banagllgysin}.

Next, we recall the notion of upward normal nonsingularity for tight embeddings.
For a closed regular embedding $V \hookrightarrow U$ of complex varieties, let $\pi \colon \operatorname{Bl}_{V}U \rightarrow U$ denote the blow-up of $U$ along $V$.
The exceptional divisor $E = \pi^{-1}(V) \subset \operatorname{Bl}_{V}U$ is the projectivization $\mathbb{P}(N)$ of the algebraic normal bundle $N$ of $V$ in $U$.

\begin{defn}[see Definition 6.5 in \cite{banagllgysin}]\label{definition upwardly normally nonsingular}
A tight embedding $Y \hookrightarrow X$ is called \emph{upwardly normally nonsingular} if the inclusion $E \subset \operatorname{Bl}_{Y \times 0}(X \times \mathbb{C})$ of the exceptional divisor $E$ is topologically normally nonsingular.
\end{defn}

According to Verdier \cite{verdierintcompl}, a regular closed algebraic embedding $g \colon Y \hookrightarrow X$ has an associated algebraic Gysin homomorphism on Borel-Moore homology
$$
g^{!}_{\alg} \colon H_{\ast}^{\operatorname{BM}} (X; \mathbb{Q}) \longrightarrow H_{\ast-2c}^{\operatorname{BM}} (Y; \mathbb{Q}),
$$
where $c$ is the complex codimension of $Y$ in $X$.

The following result is the first author's Verdier-Riemann-Roch type formula for the algebraic Gysin restriction of the Hodge-theoretic characteristic classes $IT_{1, \ast}$ (see \Cref{intersection generalized todd class}) with respect to upwardly normally nonsingular embeddings.

\begin{thm}[see Theorem 6.30 in \cite{banagllgysin}]\label{theorem Verdier-Riemann-Roch type formula}
Let $X, Y$ be pure-dimensional compact complex algebraic varieties and let $g \colon Y \hookrightarrow X$ be an upwardly normally nonsingular embedding.
Let $N = N_{Y}X$ be the algebraic normal bundle of $g$ and let $\nu$ denote the topological normal bundle of the topologically normally nonsingular inclusion underlying $g$.
Then
$$
g^{!}_{\alg}IT_{1, \ast}(X) = L^{\ast}(N) \cap IT_{1, \ast}(Y) = L^{\ast}(\nu) \cap IT_{1, \ast}(Y).
$$
\end{thm}

\subsection{Algebraic Gysin Restriction of $IT_{1, \ast}$ in a Transverse Setup} \label{Algebraic Gysin restriction in transverse setup}
Tight embeddings arise frequently from transverse intersections in ambient smooth varieties, as we shall discuss next.
For this purpose, we require the notion of Tor-independence (compare \cite[p. 1312]{banagllgysin}), which can be thought of as a transversality condition in algebraic settings.
For indications of this viewpoint in the literature, see for instance Section 1.6 in Baum-Fulton-MacPherson \cite[p. 165]{bfm2}, Definition 1.1.1 in Levine-Morel \cite[p. 1]{lm}, as well as Sierra's notion of homological transversality \cite{sierra}.

\begin{defn}\label{definition tor independence}
Closed subschemes $X, Y \subset S$ of a scheme $S$ are called \emph{Tor-independent} if
$$
\operatorname{Tor}_{i}^{\oO_{S}}(\oO_{X}, \oO_{Y}) = 0 \text{ for all } i > 0.
$$
\end{defn}

\begin{prop}[Proposition 6.3 in \cite{banagllgysin}] \label{proposition transversality implies tightness}
Let $M \hookrightarrow W$ be a closed algebraic embedding of smooth complex algebraic varieties.
Let $X \subset W$ be a (possibly singular) algebraic subvariety, equipped with an algebraic Whitney stratification and set $Y = X \cap M$.
If
\begin{itemize}
\item each stratum of $X$ is transverse to $M$, and
\item $X$ and $M$ are Tor-independent in $W$,
\end{itemize}
then the embedding $g \colon Y \hookrightarrow X$ is tight.
\end{prop}

If the embeddings $X \hookrightarrow W \hookleftarrow M$ in the previous result satisfy an even stronger transversality condition defined next, then the embedding $Y=X \cap M \hookrightarrow X$ is upwardly normally nonsingular according to \Cref{upward transversality implies upward normal nonsingularity} below.

\begin{defn}[see Definition 6.4 in \cite{banagllgysin}] \label{defn upward transversality}
Let $X \hookrightarrow W \hookleftarrow M$ be closed algebraic embeddings of algebraic varieties with $M, W$ smooth.
We say that these embeddings are \emph{upwardly transverse}, if $X$ and $M$ are Tor-independent in $W$, there exists an algebraic Whitney stratification of $X$ which is transverse to $M$ in $W$, and there exists a (possibly non-algebraic) Whitney stratification on the strict transform of $X \times \mathbb{C}$ in $\operatorname{Bl}_{M \times 0}(W \times \mathbb{C})$ which is transverse to the exceptional divisor.
\end{defn}

\begin{prop}[see Corollary 6.7 in \cite{banagllgysin}]\label{upward transversality implies upward normal nonsingularity}
If $X \hookrightarrow W \hookleftarrow M$ are upwardly transverse embeddings, then the embedding $Y = X \cap M \hookrightarrow X$ is upwardly normally nonsingular.
\end{prop}

In the transverse situation, we obtain the following extension of \Cref{theorem Verdier-Riemann-Roch type formula} by eliminating the strict transform condition.
This constitutes the main result of this section.

\begin{thm}\label{theorem gysin restriction of hodge theoretic cc}
Let $X \hookrightarrow W \hookleftarrow M$ be closed algebraic embeddings of pure-dimensional complex algebraic varieties with $M, W$ smooth, $X, M$ irreducible, and $X$ compact.
We suppose that $X$ and $M$ are Tor-independent in $W$, that $X$ is generically transverse to $M$ in $W$ (see e.g. \cite[Section 3]{bw}), and that there exists a complex algebraic Whitney stratification of $X \subset W$ which is transverse to $M$ in $W$.
Then, the embedding $g \colon Y \hookrightarrow X$ of the compact pure-dimensional subvariety $Y = X \cap M \subset X$ is tight, and the algebraic normal bundle $N = N_{Y}X$ of $g$ and the topological normal bundle $\nu$ of the topologically normally nonsingular inclusion underlying $g$ satisfy
$$
g^{!}_{\alg}IT_{1, \ast}(X) = L^{\ast}(\nu) \cap IT_{1, \ast}(Y),
$$
where $g^{!}_{\alg} \colon H_{\ast}^{\BM}(X) \rightarrow H_{\ast}^{\BM}(Y)$ denotes the algebraic Gysin homomorphism associated to $g$ constructed by Verdier in \cite{verdierintcompl}.
\end{thm}

\begin{proof}
Since $X$ admits by assumption a complex algebraic Whitney stratification which is transverse to $M$ in $W$, it follows that the product stratification on $X \times \mathbb{C}$ is a complex algebraic Whitney stratification which is transverse to $M \times 0$ in $W \times \mathbb{C}$.
Thus, \Cref{main result blowup} implies that the strict transform of $X \times \mathbb{C}$ in $\operatorname{Bl}_{M \times 0}(W \times \mathbb{C})$ admits a Whitney stratification which is transverse to the exceptional divisor.
Consequently, our transversality assumptions on $X$ and $M$ imply that the embeddings $X \hookrightarrow W \hookleftarrow M$ are upwardly transverse in the sense of \Cref{defn upward transversality}.
(Recall from \Cref{example complex algebraic Whitney stratifications are algebraic stratifications} that the complex algebraic Whitney stratification on $X$ is in particular an algebraic stratification.)
Hence, it follows from \Cref{upward transversality implies upward normal nonsingularity} that the embedding $g \colon Y \hookrightarrow X$ of the compact subvariety $Y = X \cap M \subset X$ is upwardly normally nonsingular (and, in particular, tight).
Moreover, generic transversality of $X$ and $M$ in $W$ implies that $Y$ is pure-dimensional according to \cite[Corollary 3.4]{bw}.
Finally, the algebraic Gysin restriction formula for $IT_{1, \ast}$ follows from \Cref{theorem Verdier-Riemann-Roch type formula}.
\end{proof}

\section{Algebraic versus Topological Gysin Restriction}\label{Topological and algebraic Gysin restriction}
A normally nonsingular inclusion $g \colon Y \hookrightarrow X$ (see \Cref{definition normally nonsingular embedding}) of a closed subset $Y \subset X$ with oriented normal bundle $\pi \colon E \rightarrow Y$ of rank $r$ induces on singular homology groups a topological Gysin homomorphism
$$
g^{!}_{\operatorname{top}} \colon H_{\ast}(X; \mathbb{Q}) \rightarrow H_{\ast-r}(Y; \mathbb{Q})
$$
given by the composition
$$
H_{\ast}(X) \xrightarrow{\operatorname{incl_{\ast}}} H_{\ast}(X, X \setminus Y) \xleftarrow[\cong]{e_{\ast}} H_{\ast}(E, E_{0}) \xrightarrow[\cong]{u \cap -} H_{\ast-r}(E) \xrightarrow[\cong]{\pi_{\ast}} H_{\ast-r}(Y),
$$
where $u \in H^{r}(E, E_{0})$ is the Thom class with $E_{0} = E \setminus Y$ the complement of the zero section in $E$, and $e_{\ast}$ denotes the excision isomorphism induced by the open embedding $j \colon E \rightarrow X$.

The next proposition interprets the topological Gysin map on singular homology in terms of transverse intersections.

\begin{prop}[see e.g. Proposition 2.5 in \cite{bw}] \label{proposition topological Gysin homomorphism fundamental classes}
Let $W$ be an oriented smooth manifold, $X, K \subset W$ Whitney stratified subspaces which are oriented pseudomanifolds with $K\subset X$ and $K$ compact.
Let $M\subset W$ be an oriented smooth submanifold which is closed as a subset.
Suppose that $M$ is transverse to all strata of the Whitney stratified subspaces $X \subset W$ and $K \subset W$, and that $M \cap X$ is compact.
Then the Gysin map
\[ g^!_{\operatorname{top}}: H_* (X;\rat) \longrightarrow H_{*-r} (Y;\rat) \]
associated to the normally nonsingular embedding $g: Y=M\cap X \hookrightarrow X$ (see \Cref{normally nonsingular inclusions induced by transverse intersections}), where $r$ is the (real) codimension of $Y$ in $X$, sends the fundamental class $[K]_X \in H_* (X;\rat)$ of $K$ to the fundamental class $[K\cap Y]_Y$ of the intersection $K\cap Y = M \cap K$ (which is again an oriented pseudomanifold),
\[ g^!_{\operatorname{top}} [K]_X = [K\cap Y]_Y. \]
\end{prop}

In the following, we shall discuss a variant of the topological Gysin map defined on Borel-Moore homology groups.
For reasonable compact spaces, both topological Gysin maps coincide under a natural identification of Borel-Moore homology with singular homology (see \Cref{top gysin borel moore and singular coincide} below).
This variant is needed here to show that the algebraic and topological Gysin maps coincide on algebraic cycles in a transverse setup (see \Cref{alg and top gysin} below).

Following Fulton \cite[p. 371, Eq. (1)]{fultonintth}, we employ here a construction of Borel-Moore homology defined for any topological space that can be embedded as a closed subset of some Euclidean space.
Namely, if the space $X$ is embedded as a closed subset of $\mathbb{R}^{n}$, then Fulton sets $H_{\ast}^{\BM}(X; \mathbb{Q}) = H^{n-\ast}(\mathbb{R}^{n}, \mathbb{R}^{n}-X; \mathbb{Q})$.
When $X$ is a compact ENR, then the Alexander duality isomorphism $H^{n-\ast}(\mathbb{R}^{n}, \mathbb{R}^{n}-X; \mathbb{Q}) \cong H_{\ast}(X; \mathbb{Q})$ provides a natural identification with singular homology, $H_{\ast}^{\BM}(X; \mathbb{Q}) \cong H_{\ast}(X; \mathbb{Q})$.
For a detailed discussion of this viewpoint on Borel-Moore homology, we refer to \cite[Appendix B.2, pp. 215 ff.]{fultonyoung}.

For a normally nonsingular inclusion $g \colon Y \hookrightarrow X$ of a closed subset $Y \subset X$ as above with oriented normal bundle $\pi \colon E \rightarrow Y$ of rank $r$, the topological Gysin map on Borel-Moore homology,
$$
g^{!!}_{\operatorname{top}} \colon H_{\ast}^{\BM}(X; \mathbb{Q}) \rightarrow H_{\ast-r}^{\BM}(Y; \mathbb{Q}),
$$
is given by the composition
$$
H_{\ast}^{\BM}(X) \xrightarrow{\operatorname{res}} H_{\ast}^{\BM}(U) \xrightarrow[\cong]{u \cap -} H_{\ast-r}^{\BM}(Y),
$$
where the first map restricts Borel-Moore cycles on $X$ to the open tubular neighborhood $U := j(E) \subset X$ of $Y \subset X$, and the second map is given by cap product with the Thom class $u \in H^{r}(U, U_{0}) = H^{r}(E, E_{0})$, where $U_{0} = U \setminus Y$.

The cap product used in the definition of the topological Gysin map on Borel-Moore homology is the one mentioned by Fulton in \cite[p. 371, Eq. (2)]{fultonintth}.
As explained in \cite[p. 375]{fultonintth}, its construction can be found in Fulton-MacPherson \cite[Eq. (2), p. 36]{fmp} in a more abstract setting, and involves relative \v{C}ech cohomology groups as discussed by Dold \cite[pp. 281 ff.]{dold}.
For convenience, we provide here the construction of this cap product together with some of its transformational properties (see \Cref{cor cap product fulton}).

\begin{remark}
Recall from \cite[Definition 6.1, p. 281]{dold} that the relative \v{C}ech cohomology groups $\check{H}^i (A, B)$ are defined for any pair $B \subset A$ of locally compact subspaces of an ENR.
Furthermore, for any locally compact subspaces $A_{1}, A_{2}$ of an ENR which are separated by $A_{1} \cap A_{2}$ (e.g. when $A_{1}, A_{2}$ are both open or closed in $A_{1} \cup A_{2}$), there is an excision isomorphism $\check{H}^{\ast}(A_{1} \cup A_{2}, A_{1}) \cong \check{H}^{\ast}(A_{2}, A_{1} \cap A_{2})$ by \cite[p. 286, Eq. (6.16)]{dold}.
\end{remark}

\begin{prop}\label{prop cap product construction chech and borel moore}
For a topological space $X$ that can be embedded as a closed subset of some Euclidean space, the following hold:
\begin{enumerate}
\item For any closed subset $Y \subset X$, there is a cap product
\[ \check{H}^i (X,X-Y) \times H^\BM_k (X)
 \stackrel{\cap}{\longrightarrow}
  H^\BM_{k-i} (Y). \]     
  
\item Every open subset $X' \subset X$ can be embedded as a closed subset of some Euclidean space, and if $X'$ is an open neighborhood of a closed subset $Y \subset X$, then we have a commutative diagram
 \begin{equation*}
 \xymatrix@C=2pt{ 
\check{H}^i (X,X-Y) \ar[d]_{\cong}^{\rho^*} & \times & H^\BM_{k} (X) \ar[d]^{r^{\ast}} \ar[rrrrr]^\cap & &&&& H^\BM_{k-i} (Y) \ar[d]^{=} \\
\check{H}^i (X', X'-Y) & \times & H^\BM_{k} (X') \ar[rrrrr]^\cap & &&&&
  H^\BM_{k-i} (Y),
}
\end{equation*}
where $\rho^{\ast}$ is the excision isomorphism induced by inclusion in \v{C}ech cohomology, and $r^{\ast}$ is restriction of a Borel-Moore cycle to an open subset.  

\item Any cartesian diagram of closed embeddings
\[ \xymatrix{
Y' = X'\cap Y \ar[d]_{g'} \ar[r]^-{f_{Y}} & Y \ar[d]^{g} \\
X' \ar[r]_{f_{X}} & X  
} \]
induces compatible cap products
 \begin{equation*}
 \xymatrix@C=2pt{
\check{H}^i (X',X'-Y') & \times & H^\BM_{k}(X') \ar[d]^{f_{X \ast}} \ar[rrrrrr]^\cap &&&&&& H^\BM_{k-i}(Y') \ar[d]^{f_{Y \ast}} \\
\check{H}^i (X, X-Y)  \ar[u]^{f_{X}^{\ast}} & \times & H^\BM_{k} (X) \ar[rrrrrr]^\cap &&&&&& H^\BM_{k-i} (Y),
}
\end{equation*}
where $f_{X}^{\ast}$ is the induced map on \v{C}ech cohomology, and $f_{X \ast}, f_{Y \ast}$ are induced on Borel-Moore homology by the proper maps $f_{X}, f_{Y}$.
\end{enumerate}
\end{prop}

\begin{proof}
Suppose that $X$ is embedded as a closed subset of $\mathbb{R}^{n}$.
For the construction of the cap product in claim (1), let $(U,V) \supset (X,X-Y)$ be an open neighborhood pair in $\mathbb{R}^{n}$.
Since both $V$ and $U-X$ are open in $U$, singular cohomology comes with a standard cup product
\[ H^i (U,V) \times H^j (U,U-X)
 \stackrel{\cup}{\longrightarrow}
  H^{i+j} (U, V \cup (U-X)). \]
Now it follows from $X - Y \subset V$ that $U-Y = (X-Y) \cup (U-X) \subset V \cup (U-X)$.
Hence, there is a restriction homomorphism
\[   H^{i+j} (U, V \cup (U-X)) \longrightarrow
   H^{i+j} (U, U-Y). \]
Composing with it, we obtain a cup product\footnote{In Fulton-MacPherson \cite[Eq. (2), p. 36]{fmp}, the additional asumption $V\cap X = X-Y$ is imposed.}
\[ H^i (U,V) \times H^j (U,U-X)
 \stackrel{\cup}{\longrightarrow}
  H^{i+j} (U, U-Y). \]   
Using excision, this can be rewritten as a product
\[ H^i (U,V) \times H^j (\mathbb{R}^{n},\mathbb{R}^{n}-X)
 \stackrel{\cup}{\longrightarrow}
  H^{i+j} (\mathbb{R}^{n}, \mathbb{R}^{n}-Y). \]   
Next, we note that $X-Y \subset X$ are locally compact subsets of the ENR $\mathbb{R}^{n}$ by \cite[Lemma 8.3, p. 80]{dold} because they are locally closed subsets by assumption.
Therefore, the relative \v{C}ech cohomology groups
\[ \check{H}^i (X,X-Y) := \colim \{ H^i (U,V) ~|~
             (U,V) \supset (X,X-Y)  \} \]
are defined as in \cite[Definition 6.1, p. 281]{dold}.
Consequently, there is a cup product 
\[ \check{H}^i (X,X-Y) \times H^j (\mathbb{R}^{n},\mathbb{R}^{n}-X)
 \stackrel{\cup}{\longrightarrow}
  H^{i+j} (\mathbb{R}^{n}, \mathbb{R}^{n}-Y). \] 
By the definition of Borel-Moore homology used by Fulton, we have
\[ H^j (\mathbb{R}^{n},\mathbb{R}^{n}-X) = H^\BM_{n-j} (X), \quad
    H^{i+j} (\mathbb{R}^{n},\mathbb{R}^{n}-Y) = H^\BM_{n-i-j} (Y). \]
We arrive at the desired cap product
\[ \check{H}^i (X,X-Y) \times H^\BM_k (X)
 \stackrel{\cap}{\longrightarrow}
  H^\BM_{k-i} (Y). \]
  
As for claim (2), let $X' \subset X$ be an open neighborhood of the closed subset $Y \subset X$.
Then, there exists an open subset $W \subset \mathbb{R}^{n}$ such that $X' = W \cap X$.
We note that $X'$ is a closed subset of the manifold $W$, which can be embedded as a closed subset of some Euclidean space.
If $(U,V) \supset (X,X-Y)$ is an open neighborhood pair in $\mathbb{R}^{n}$, then $(W \cap U, W \cap V) \supset (X',X'-Y)$ is an open neighborhood pair in $W$, and the diagram of restrictions
\[ \xymatrix@C=2pt{
H^i (U,V) \ar[d] & \times & H^j (\mathbb{R}^{n},\mathbb{R}^{n}-X) \ar[d]
  \ar[rrrrrr]^\cup &&&&&&
  H^{i+j} (\mathbb{R}^{n}, \mathbb{R}^{n}-Y) \ar[d] \\
H^i (W \cap U, W \cap V) & \times 
  & H^j (W,W-X') \ar[rrrrrr]^\cup &&&&&&
  H^{i+j} (W, W-Y)
} \]  
commutes by naturality of the cup product.
Finally, we obtain the desired commutative diagram by passing to the colimits to obtain the induced map on \v{C}ech cohomology (see \cite[Definition 6.3, p. 282]{dold}), and by definition of the restriction map in Borel-Moore homology (see \cite[Eq. (30), p. 218]{fultonyoung}).

To show claim (3), we observe that any open neighborhood pair $(U,V) \supset (X,X-Y)$ in $\mathbb{R}^{n}$ induces compatible cup products
\[ \xymatrix@C=2pt{
H^i (U,V) & \times & H^j (U,U-X') \ar[d] \ar[rrrrrr]^\cup &&&&&& H^{i+j} (U, U-Y') \ar[d] \\
H^i (U, V)  \ar[u] & \times & H^j (U,U-X) \ar[rrrrrr]^\cup &&&&&& H^{i+j} (U, U-Y),
} \]
where the vertical arrows are induced by inclusion.
The claim now follows by passing to  the colimits to obtain the induced map on \v{C}ech cohomology (see \cite[Definition 6.3, p. 282]{dold}), and by definition of the pushforward for proper maps in Borel-Moore homology (see \cite[p. 218]{fultonyoung}).
\end{proof}

\begin{remark}\label{remark borel moore gysin fundamental classes}
In \Cref{proposition topological Gysin homomorphism fundamental classes}, the assumption that $K$ is compact can be dropped when using the topological Gysin map on Borel-Moore homology instead of singular homology.
The proof is very similar, but uses base change as stated in \Cref{prop cap product construction chech and borel moore}(3), as well as the fact that the topological Gysin map in Borel-Moore homology maps the fundamental class to the fundamental class.
\end{remark}

\begin{cor}\label{cor cap product fulton}
For any closed subset $Y$ of an ENR $X$, we have a cap product
\[ H^i (X,X-Y) \times H^\BM_k (X)
 \stackrel{\cap}{\longrightarrow}
  H^\BM_{k-i} (Y) \]  
 with the properties (2) and (3) of \Cref{prop cap product construction chech and borel moore} with respect to singular cohomology.
\end{cor}

\begin{proof}
Since $X$ is an ENR, we can embed it as a closed subset of some Euclidean space $\mathbb{R}^{n}$ according to \cite[Proposition 8.1 and Lemma 8.2, p. 80]{dold}, so that the previous proposition yields a cap product using \v{C}ech cohomology. 
By assumption, $X$ is an ENR, and hence the open subset $X - Y \subset X$ is an ENR as well.
By \cite[Proposition 6.12, p. 285]{dold}, $\check{H}^i (X,X-Y)$ is therefore naturally isomorphic to singular cohomology $H^i (X,X-Y)$.
\end{proof}

\begin{remark}
For $Y = X$ an ENR, the cap product of \Cref{cor cap product fulton} specializes to a cap product of the form $H^{\ast}(X) \times H_{\ast}^{\BM}(X) \rightarrow H_{\ast}^{\BM}(X)$.
This is the cap product that appears in \Cref{theorem Verdier-Riemann-Roch type formula} and \Cref{theorem gysin restriction of hodge theoretic cc}.
When $X$ is a compact CW complex, this cap product corresponds to the ordinary cap product under the natural identification of Borel-Moore homology with singular homology.
(In fact, if $\alpha \in H^{i}(X; \mathbb{Q}) = \check{H}^{i}(X; \mathbb{Q})$ is represented by $\alpha_{U} \in H^{i}(U; \mathbb{Q})$ on some open neighborhood $U$ of $X \subset \mathbb{R}^{n}$, then there is a commutative diagram
\[ \xymatrix{
H_{n-(i+j)}(U, U-X) \ar[r]^{\alpha_{U} \cap -} & H_{n-j}(U, U-X) \\
H^{i+j}(U) \ar[u]^{- \cap \mu_{X}} \ar[r]_{\alpha_{U} \cup -} & H^{j}(U)  \ar[u]_{- \cap \mu_{X}},
} \]
where $\mu_{X} \in H_{n}(U, U-X; \mathbb{Q})$ denotes the orientation class of $X \subset U$.
Passing to \v{C}ech cohomology of $X$ in the lower horizontal row induces the ordinary cup product map
$$
\alpha \cup - \colon H^{n-(i+j)}(X) \rightarrow H^{n-j}(X),
$$
and the vertical maps become Alexander duality isomorphisms $H^{\ast}(X) \cong H_{n-\ast}(U, U-X)$ (see e.g. Bredon \cite[Theorem 8.3, p. 351]{bre}).
Then, the claim follows by applying $\operatorname{Hom}_{\mathbb{Q}}(-, \mathbb{Q})$ to the resulting commutative diagram, and identifying $H^{\ast}(-) \cong \operatorname{Hom}_{\mathbb{Q}}(H_{\ast}(-), \mathbb{Q})$ via the Kronecker pairing.
Note that the involved $\mathbb{Q}$-vector spaces are finite dimensional.)
\end{remark}

\begin{prop}\label{top gysin borel moore and singular coincide}
Let $(X, Y)$ be a compact CW pair such that $X$ is endowed with a topologically cone-like topological stratification.
If the inclusion $g \colon Y \hookrightarrow X$ is normally nonsingular with oriented normal bundle $\nu \colon E \rightarrow Y$ of rank $r$, the topological Gysin maps
\[ g^{!}_{\operatorname{top}}, g^{!!}_{\operatorname{top}}: H_{\ast}(X) \longrightarrow H_{\ast-r} (Y)   \]
coincide under the natural identification of Borel-Moore homology with singular homology.
\end{prop}

\begin{proof}
We may identify a closed tubular neighborhood of $Y \subset X$ with the total space of the disc bundle $D\nu$ of the topological normal bundle $\nu$ of the normally nonsingular inclusion $g \colon Y \hookrightarrow X$. Let $S\nu \subset D\nu$ denote the sphere bundle and
$D^\circ \nu = D\nu - S\nu$ the open disc bundle.
Furthermore, $Z = X - D^{\circ} \nu$ is a closed subset of $X$.
The ENR $X$ can be embedded a closed subset of some Euclidean space $\mathbb{R}^{n}$.
Then, we obtain a commutative diagram
\[ \xymatrix{
H_{n-k}(\mathbb{R}^{n} - S\nu, \mathbb{R}^{n} - D\nu) & H_{n-k}(\mathbb{R}^{n} - Z, \mathbb{R}^{n} - X) \ar[r]^-{\operatorname{incl}^{\ast}} \ar[l]_-{\operatorname{incl}^{\ast}}^-{\cong} & H_{n-k}(\mathbb{R}^{n}, \mathbb{R}^{n} - X) \\
H^{k}(D\nu, S\nu) \ar[u]^{\cong}_{- \cap \vartheta} & H^{k}(X, Z) \ar[u]^{\cong}_{- \cap \vartheta} \ar[r]_{\operatorname{incl}^{\ast}} \ar[l]^{\operatorname{incl}^{\ast}}_{\cong} & H^{k}(X)\ar[u]^{\cong}_{- \cap \vartheta}
} \]
and for any open neighborhood $U$ of $D \nu \subset \mathbb{R}^{n}$ a commutative diagram
\[ \xymatrix{
H_{n-k}(U- S\nu, U - D\nu) \ar[r]^-{\operatorname{incl}_{\ast}}_-{\cong} & H_{n-k}(\mathbb{R}^{n} - S\nu, \mathbb{R}^{n} - D\nu) \\
H^{k}(D\nu, S\nu) \ar[u]^{\cong}_{- \cap \vartheta|_{U}} \ar[r]^{=} & H^{k}(D\nu, S\nu) \ar[u]^{\cong}_{- \cap \vartheta},
} \]
where the vertical maps are Alexander duality isomorphisms described by cap product with an orientation $\vartheta$ of $\mathbb{R}^{n}$, or its restriction to $U$ (see e.g. Bredon \cite[Theorem 8.3, p. 351]{bre}, who also provides a description of the Alexander duality isomorphisms on the (co)chain level in \cite[p. 349]{bre}).

The Thom class $\tau(\nu) \in H^{r}(D\nu, S\nu) = \check{H}^{r}(D\nu, S\nu)$ can be represented by an element $t \in H^{r}(U, V)$ for some open neighborhood pair $(U, V)$ of $(D\nu, S\nu)$ in $\mathbb{R}^{n}$.
By making $V$ smaller if necessary, we may assume without loss of generality that the inclusion $(U, V) \subset (U, U - Y)$ is a homotopy equivalence of pairs.
Let $u \in H^{r}(U, U-Y)$ be the element that restricts to $t$.
Writing $W = U-S\nu$, we obtain the commutative diagram
\[ \xymatrix{
H_{n-(k-r)}(W, W-Y) \ar[rr]^{u|_{W} \cap -} & & H_{n-k}(W, W-D^{\circ}\nu) = H_{n-k}(U-S\nu, U-D\nu) \\
H^{k-r}(Y) \ar[u]^{\cong}_{- \cap \vartheta|_{W}} & H^{k-r}(D \nu) \ar[l]^{\operatorname{incl}^{\ast}}_{\cong} \ar[r]_{\tau(\nu) \cup -} & H^{k}(D\nu, S\nu) \ar[u]^{\cong}_{- \cap \vartheta|_{U}}
} \]
where the vertical maps are Alexander duality isomorphisms described by cap product with the orientation $\vartheta|_{W}$ of $W$ or the orientation $\vartheta|_{U}$ of $U$.
Again, one checks commutativity of the diagram by using the description of the Alexander duality isomorphisms on the (co)chain level provided in \cite[p. 349]{bre}.

Finally, the claim follows by applying $\operatorname{Hom}_{\mathbb{Q}}(-, \mathbb{Q})$ to the concatenation of the above three commutative diagrams, and identifying $H^{\ast}(-) \cong \operatorname{Hom}_{\mathbb{Q}}(H_{\ast}(-), \mathbb{Q})$ via the Kronecker pairing.
Note that the involved $\mathbb{Q}$-vector spaces are finite dimensional.
\end{proof}

Let $A_* (Z)$ denote the Chow groups of a complex variety $Z$ and
let $\cl = \cl_{Z}: A_* (Z) \to H^\BM_{2*} (Z)$ be the cycle map from
Chow homology to rational Borel-Moore homology.
Let
\[ u_{f} \in H^{2d} (X, X-Y) \]  
be the orientation class of the
regular closed embedding $f \colon Y \hookrightarrow X$
as in \cite[19.2, p. 378]{fultonintth}.  

The cap product of \Cref{cor cap product fulton} can be used to evaluate the orientation class of a closed regular embedding on a cycle class in Borel-Moore homology as follows.

\begin{thm}[see Theorem 19.2 in Fulton \cite{fultonintth}] \label{theorem gysin restriction of cycles is cap product with orientation class}
Let $f \colon Y \hookrightarrow X$ be a regular embedding of complex algebraic varieties of codimension $d$.
Then, for all $k$-cycles $\alpha$ on $X$,
$$
cl_{Y}(f^{!}_{\operatorname{alg}}\alpha) = u_{f} \cap cl_{X}(\alpha)
$$
in $H_{2k-2d}^{BM}(Y)$.
\end{thm}

\begin{remark}
More generally, Theorem 19.2 in \cite{fultonintth} uses the refined Gysin homomorphism \cite[Section 6.2]{fultonintth} associated to the fiber square formed by the regular embedding $f \colon Y \hookrightarrow X$ and a morphism $g \colon X' \rightarrow X$.
However, for $g = \operatorname{id}_{X}$ the refined Gysin homomorphism coincides with Verdier's algebraic Gysin homomorphism $f^{!}_{\operatorname{alg}}$ for regular embeddings (see \cite[p. 117]{fultonintth}, where we note that $f^{!}_{\operatorname{alg}} = f^{\ast}$ in Fulton's notation).
\end{remark}

\begin{thm}\label{alg and top gysin}
Let $W$ be a smooth complex algebraic variety, 
$M\subset W$ a smooth closed subvariety and
$X\subset W$ a closed subvariety such that an algebraic Whitney stratification
of $X$ in $W$ is transverse to $M$ and
$X$ and $M$ are Tor-independent in $W$.
Then the algebraic and topological Gysin restriction homomorphisms
associated to the inclusion $g: Y = X\cap M \hookrightarrow X,$
\[ g^{!}_{\alg}, g^{!!}_{\operatorname{top}}: H^\BM_* (X) \longrightarrow H^\BM_{*-2d} (Y)   \]
coincide on algebraic cycles, where $d$ is the
complex codimension of the embedding $g$.
\end{thm}
\begin{proof}
By smoothness, the closed embedding $M\subset W$ is regular
with algebraic normal bundle $N_M W$.
The Tor-independence of $X$ and $M$ ensures that the closed embedding
$g: Y\hookrightarrow X$ is regular as well,
and that the excess normal bundle vanishes, i.e. the
canonical closed embedding $N_Y X \to j^* N_M W$ is an isomorphism
of algebraic vector bundles, where $j$ is the embedding
$j: Y \hookrightarrow M$.
According to Verdier \cite[p. 222, 9.2.1]{verdierintcompl},
the algebraic Gysin map of a closed regular embedding commutes with the
cycle map from Chow to Borel-Moore homology. Thus there is
a commutative diagram
\begin{equation}\label{dia.verdiershowbmcyclemapgysin}
\xymatrix{
A_{\ast} (X) \ar[d]_\cl \ar[r]^{g^{!}_{\alg}} &
  A_{\ast-d} (Y) \ar[d]^\cl \\
H^\BM_{2\ast} (X) \ar[r]^{g^{!}_{\alg}} &
  H^\BM_{2 \ast-2d} (Y).  
} 
\end{equation}
Since $X$ and $M$ are transverse and Tor-independent in $W$,
\Cref{proposition transversality implies tightness}
implies that the embedding $g: Y \hookrightarrow X$ is
tight, i.e. the underlying topological embedding (in the
analytic topology) is normally nonsingular with tubular
neighborhood described by the disc bundle $D\nu$ of a topological
normal bundle $\nu$ which is isomorphic to the underlying topological vector
bundle of the algebraic normal bundle $N_Y X$.
Let $S\nu \subset D\nu$ denote the sphere bundle and
$D^\circ \nu = D\nu - S\nu$ the open disc bundle. 
Then, \Cref{prop cap product construction chech and borel moore} yields a commutative diagram
\begin{equation}\label{borel moore cap product restriction}
\xymatrix@C=2pt{
H^i (X,X-Y) \ar[d]_\cong^{\rho^*} & \times & H^\BM_{p-j} (X) \ar[d]^{r^*}
  \ar[rrrrrr]^\cap &&&&&&
  H^\BM_{p-i-j} (Y) \ar[d]^{=} \\
H^i (D^\circ \nu, (D^\circ \nu)-Y) & \times 
  & H^\BM_{p-j} (D^\circ \nu) \ar[rrrrrr]^\cap &&&&&&
  H^\BM_{p-i-j} (Y),
}
\end{equation}
where the left vertical arrow is the excision isomorphism induced by inclusion, and the middle vertical arrow is restriction of a Borel-Moore cycle to an open subset.

Let
\[ u_g \in H^{2d} (X, X-Y) \]  
be the orientation class of the
regular closed embedding $g: Y \hookrightarrow X$
as in \cite[19.2, p. 378]{fultonintth}.  
We shall next compute this class.
Consider the cartesian diagram of closed embeddings
\[ \xymatrix{
Y = M\cap X \ar[d]_j \ar[r]^-g &
  X \ar[d]^i \\
M \ar[r]_f & W.  
} \]
By \cite[Lemma 19.2 (a), p. 379]{fultonintth}, applied
to the above cartesian diagram containing the regular
embeddings $f$ and $g$,
\[ i^* (u_f)= u_g \in H^{2d} (X,X-Y), \qquad
   i^*: H^{2d} (W,W-M) \to H^{2d} (X,X-Y). \]
Let $\nu_M$ be the underlying topological vector bundle
of the algebraic normal bundle $N_M W$.
By tightness of the embedding $g \colon Y \hookrightarrow X$, there
is an isomorphism $\nu = j^* \nu_M$, $j:Y\hookrightarrow M$ inherited from the isomorphism of algebraic vector bundles.
Let $i_D: D\nu \to D\nu_M$ denote the bundle map
covering $j$. Then by naturality of the Thom class
\[ t(\nu) = i^*_D t(\nu_M). \]

Consider the following factorization of the above cartesian diagram:
\[ \xymatrix{
Y \ar[d]_j \ar[r]^0 & D\nu \ar@{^{(}->}[r]^\rho \ar[d]_{i_D} &
  X \ar[d]^i \\
M \ar[r]^0 & D\nu_M \ar@{^{(}->}[r]^\delta & W.  
} \]
By \cite[p. 372, bottom]{fultonintth},
\[ \delta^* (u_f) = t(\nu_M), \]
since $f$ is a closed embedding of nonsingular varieties.
Therefore,
\[
\rho^* (u_g) = \rho^* i^* (u_f) = \i^*_D \delta^* (u_f) =
   i^*_D t(\nu_M) = t(\nu).
\]
By \Cref{theorem gysin restriction of cycles is cap product with orientation class},
\begin{equation} \label{equ.fultonexample1921}
u_g \cap \cl (\alpha) = \cl (g^{!}_{\alg} \alpha)
\end{equation}
for any algebraic cycle $\alpha \in A_* (X)$.  
Using Verdier's diagram (\ref{dia.verdiershowbmcyclemapgysin}),
and the commutativity of the cap diagram (\ref{borel moore cap product restriction}), we obtain for
$\alpha \in A_* (X),$
\begin{align*}
g^{!}_{\alg} (\cl (\alpha))  
&= \cl (g^{!}_{\alg} \alpha) \\
&= u_g \cap \cl (\alpha) \\
&= \rho^{\ast} (u_{g}) \cap r^{\ast} (\cl(\alpha)) \\
&= t(\nu) \cap r^{\ast} (\cl (\alpha)).
\end{align*}  
Thus for any class $\beta \in H^\BM_{\ast} (X)$ in the image of the cycle map,
\[ g^{!}_{\alg}(\beta) = t(\nu) \cap r^{\ast} \beta = g^{!!}_{\operatorname{top}}(\beta). \]  
\end{proof}

\section{Gysin Coherent Characteristic Classes}\label{Gysin coherent characteristic classes}
In this section, we recall the notion of Gysin coherent characteristic classes (see \Cref{definition gysin coherent characteristic classes} below), and state the uniqueness theorem, which is the main result of \cite{bw} (see \Cref{main result on Gysin coherent characteristic classes with respect to x} below).
In the present paper, we shall discuss algebraic characteristic classes such as Todd classes (see \Cref{example todd class}), Chern classes (see \Cref{example chern class}), as well as motivic Hodge classes (see \Cref{the class it}) within the framework of Gysin coherence.

In the following, by a variety we mean a pure-dimensional complex quasiprojective algebraic variety.

Let $\mathcal{X}$ be a family of inclusions $i \colon X \rightarrow W$, where $W$ is a smooth variety, and $X \subset W$ is a compact irreducible subvariety.
We require the following properties for $\mathcal{X}$:
\begin{itemize}
\item For every Schubert subvariety $X \subset G$ of a Grassmannian $G$, the inclusion $X \rightarrow G$ is in $\mathcal{X}$.
\item If $i \colon X \rightarrow W$ and $i' \colon X' \rightarrow W'$ are in $\mathcal{X}$, then the product $i \times i' \colon X \times X' \rightarrow W \times W'$ is in $\mathcal{X}$.
\item Given inclusions $i \colon X \rightarrow W$ and $i' \colon X' \rightarrow W'$ of compact subvarieties in smooth varieties, and an isomorphism $W \stackrel{\cong}{\longrightarrow} W'$ that restricts to an isomorphism $X \stackrel{\cong}{\longrightarrow} X'$, it follows from $i \in \mathcal{X}$ that $i' \in \mathcal{X}$.
\item For all closed subvarieties $X \subset M \subset W$ such that $X$ is compact and $M$ and $W$ are smooth, it holds that if the inclusion $X \rightarrow M$ is in $\mathcal{X}$, then the inclusion $X \rightarrow W$ is in $\mathcal{X}$.
\end{itemize}

Next, for a given family $\mathcal{X}$ of inclusions as above, by \emph{$\mathcal{X}$-transversality}, we mean a symmetric relation for closed irreducible subvarieties of a smooth variety that satisfies the following properties:
\begin{itemize}
\item The intersection $Z \cap Z'$ of two $\mathcal{X}$-transverse closed irreducible subvarieties $Z, Z' \subset W$ of a smooth variety $W$ is \emph{proper}, that is, $Z \cap Z'$ is pure-dimensional of codimension $c + c'$, where $c$ and $c'$ are the codimensions of $Z$ and $Z'$ in $W$, respectively.
\item The following analog of Kleiman's transversality theorem holds for the action of $GL_{n}(\mathbb{C})$ on the Grassmannians $G = G_{k}(\mathbb{C}^{n})$.
If $i \colon X \rightarrow G$ and $i' \colon X' \rightarrow G$ are inclusions in $\mathcal{X}$, then there is a nonempty open dense subset $U \subset GL_{n}(\mathbb{C})$ (in the complex topology) such that $X$ is $\mathcal{X}$-transverse to $g \cdot X'$ for all $g \in U$.
\item Locality: If $Z, Z' \subset W$ are $\mathcal{X}$-transverse closed irreducible subvarieties of a smooth variety $W$ and $U \subset W$ is a (Zariski) open subset that has nontrivial intersections with $Z$ and $Z'$, then $Z \cap U$ and $Z' \cap U$ are $\mathcal{X}$-transverse in $U$.
\end{itemize}

\begin{example}\label{example typical choice for family x and transversality}
The family $\mathcal{X}_{0}$ consisting of all inclusions of compact irreducible subvarieties in smooth varieties satisfies the above family requirements.
Furthermore, a notion of $\mathcal{X}_{0}$-transversality is obtained by calling two closed irreducible subvarieties $Z, Z' \subset W$ of a smooth variety $W$ $\mathcal{X}_{0}$-transverse if they admit complex algebraic Whitney stratifications (see \Cref{example complex algebraic Whitney stratifications are algebraic stratifications}) such that every stratum of $Z$ is transverse to every stratum of $Z'$ (as smooth submanifolds of $W$).
This uses Kleiman's transversality theorem \cite{kleiman}.
\end{example}

Recall from \Cref{Topological and algebraic Gysin restriction} that every inclusion $f \colon M \rightarrow W$ of a smooth closed subvariety $M$ of (complex) codimension $c$ in a smooth variety $W$ induces a topological Gysin map on singular homology, $f^{!}_{\operatorname{top}} \colon H_{\ast}(W; \mathbb{Q}) \rightarrow H_{\ast - 2c}(M; \mathbb{Q})$.

\begin{defn}\label{definition gysin coherent characteristic classes}
A \emph{Gysin coherent characteristic class $c\ell$ with respect to $\mathcal{X}$} is a pair
$$
c\ell = (c\ell^{\ast}, c\ell_{\ast})
$$
consisting of a function $c\ell^{\ast}$ that assigns to every inclusion $f \colon M \rightarrow W$ of a smooth closed subvariety $M \subset W$ in a smooth variety $W$ an element
$$
c\ell^{\ast}(f) = c\ell^{0}(f) + c\ell^{1}(f) + c\ell^{2}(f) + \dots \in H^{\ast}(M; \mathbb{Q}), \quad c\ell^{p}(f) \in H^{p}(M; \mathbb{Q}),
$$
with $c\ell^{0}(f) = 1$, and a function $c\ell_{\ast}$ that assigns to every inclusion $i \colon X \rightarrow W$ of a compact possibly singular subvariety $X \subset W$ of complex dimension $d$ in a smooth variety $W$ an element
$$
c\ell_{\ast}(i) = c\ell_{0}(i) + c\ell_{1}(i) + c\ell_{2}(i) + \dots + c\ell_{2d}(i) \in H_{\ast}(W; \mathbb{Q}), \quad c\ell_{p}(i) \in H_{p}(W; \mathbb{Q}),
$$
with $c\ell_{2d}(i) = [X]_{W}$, such that the following properties hold:
\begin{enumerate}
\item\label{axiom multiplicativity} \emph{(Multiplicativity)}
For every $i \colon X \rightarrow W$ and $i' \colon X' \rightarrow W'$, we have
$$
c\ell_{\ast}(i \times i') = c\ell_{\ast}(i) \times c\ell_{\ast}(i').
$$
\item\label{axiom isomorphism} \emph{(Isomorphism invariance)}
For every $f \colon M \rightarrow W$ and $f' \colon M' \rightarrow W'$, and every isomorphism $W \stackrel{\cong}{\longrightarrow} W'$ that restricts to an isomorphism $\phi \colon M \stackrel{\cong}{\longrightarrow} M'$, we have
$$
\phi^{\ast}c\ell^{\ast}(f') = c\ell^{\ast}(f).
$$
Moreover, for every $i \colon X \rightarrow W$ and $i' \colon X' \rightarrow W'$, and every isomorphism $\Phi \colon W \stackrel{\cong}{\longrightarrow} W'$ that restricts to an isomorphism $X \stackrel{\cong}{\longrightarrow} X'$, we have
$$
\Phi_{\ast}c\ell_{\ast}(i) = c\ell_{\ast}(i').
$$
\item\label{axiom locality} \emph{(Naturality)}
For every $i \colon X \rightarrow W$ and $f \colon M \rightarrow W$ such that $X \subset M$, the inclusion $i^{M} := i| \colon X \rightarrow M$ satisfies
$$
f_{\ast}c\ell_{\ast}(i^{M}) = c\ell_{\ast}(i).
$$
\item\label{axiom gysin} \emph{(Gysin restriction in a transverse setup)}
There exists a notion of $\mathcal{X}$-transversality such that the following holds.
For every inclusion $i \colon X \rightarrow W$ in $\mathcal{X}$ and every inclusion $f \colon M \rightarrow W$ such that $M$ is irreducible, and $M$ and $X$ are $\mathcal{X}$-transverse in $W$, the inclusion $j \colon Y \rightarrow M$ of the pure-dimensional compact subvariety $Y := M \cap X \subset M$ satisfies
$$
f^{!}_{\operatorname{top}} c\ell_{\ast}(i) = c\ell^{\ast}(f) \cap c\ell_{\ast}(j).
$$
\end{enumerate}
Such a class $c\ell$ is called \emph{Gysin coherent characteristic class} if $\mathcal{X} = \mathcal{X}_{0}$ is the family of all inclusions of compact irreducible subvarieties in smooth varieties (compare \Cref{example typical choice for family x and transversality}).
\end{defn}

The genus $|c\ell_{\ast}|$ of a Gysin coherent characteristic class $c\ell$ with respect to $\mathcal{X}$ is defined as the composition of $c\ell_{\ast}$ with the homological augmentation, $|c\ell_{\ast}| = \varepsilon_{\ast}c\ell_{\ast} \in \mathbb{Q}$.

\begin{example}\label{example l class as gysin coherent cc}
In \cite[Section 9]{bw}, the framework of Gysin coherence with respect to the family $\mathcal{X}_{0}$ of \Cref{example typical choice for family x and transversality} is applied to Goresky-MacPherson $L$-classes (see \Cref{example l class as gysin coherent cc} below), where $\mathcal{X}_{0}$-transversality is chosen to mean simultaneously topological transversality with respect to appropriate Whitney stratifications and generic transversality.
It is shown that the pair $(c\ell^{\ast}, c\ell_{\ast})$ given by $c\ell^{\ast}(f \colon M \hookrightarrow W) = L^{\ast}(\nu_{M \subset W})$ and $c\ell_{\ast}(i \colon X \hookrightarrow W) = i_{\ast}L_{\ast}(X)$, where $L^{\ast}$ is Hirzebruch's cohomological $L$-class of a vector bundle and $L_{\ast}$ is the Goresky-MacPherson $L$-class, forms a Gysin coherent characteristic class.
The associated genus is the signature $\sigma(X)$ of the Goresky-MacPherson intersection form on middle-perversity intersection homology of $X$.
\end{example}

The main result of \cite{bw} is the following

\begin{thm}[Uniqueness Theorem]\label{main result on Gysin coherent characteristic classes with respect to x}
Let $c\ell$ and $\widetilde{c\ell}$ be Gysin coherent characteristic classes with respect to $\mathcal{X}$.
If $c\ell^{\ast} = \widetilde{c\ell}^{\ast}$ and $|c\ell_{\ast}| = |\widetilde{c\ell}_{\ast}|$ for the associated genera, then we have $c\ell_{\ast}(i) = \widetilde{c\ell}_{\ast}(i)$ for all inclusions $i \colon X \rightarrow G$ in $\mathcal{X}$ of compact irreducible subvarieties in ambient Grassmannians.
\end{thm}

\begin{remark}\label{remark top class only for irreducible varieties}
An inspection of the proof of the uniqueness theorem above shows that it suffices to request the assumption $c \ell_{2d}(i) = [X]_{W}$ in the definition of Gysin coherent classes $c \ell$ only for irreducible $X$.
In fact, this assumption is only used in the proof of Theorem 7.1 and in the proof of Theorem 6.4 in \cite{bw}, where the varieties under consideration are irreducible.
\end{remark}

\section{The class $IT_{1, \ast}$ as a Gysin Coherent Characteristic Class}\label{the class it}

The class $IT_{1, \ast}$ fits into the framework of Gysin coherent characteristic classes (\Cref{definition gysin coherent characteristic classes}), as we shall show in \Cref{proposition it class is l type characteristic class} below.
Our method of proof requires moving varieties into transverse position by means of suitable generalizations of Kleiman's transversality theorem.
Such generalizations are known in the Cohen-Macaulay context (see Sierra \cite{sierra}).
Rational and normal Gorenstein singularities are Cohen-Macaulay.
For instance, Schubert varieties as well as toric varieties have rational singularities, and hence are Cohen-Macaulay.

As in \Cref{Gysin coherent characteristic classes}, by a variety we mean a pure-dimensional complex quasiprojective algebraic variety.
For later use in the proof of \Cref{proposition it class is l type characteristic class} below, we record the following

\begin{prop}\label{prop top component of it1}
For an irreducible projective variety $X$ of complex dimension $d$, the highest non-trivial homogeneous component of the class
$$
IT_{1, \ast}(X) = IT_{1, 0}(X) + IT_{1, 1}(X) + \dots \in H_{\ast}(X; \mathbb{Q})
$$
is the fundamental class, $IT_{1, 2d}(X) = [X]_{X}$.
\end{prop}

\begin{proof}
Since $X$ is compact and irreducible, we know that $H_{2d}(X; \mathbb{Q})$ is generated by the fundamental class $[X]_{X}$, and that $H_{i}(X; \mathbb{Q}) = 0$ for $i > 2d$.
Therefore, writing $IT_{1, 2d}(X) = r \cdot [X]_{X}$ for some $r \in \mathbb{Q}$, it remains to show that $r = 1$.
For this purpose, we fix an embedding $X \subset \mathbb{P}^{m}$ and a complex algebraic Whitney stratification on $X$ (see \Cref{example complex algebraic Whitney stratifications are algebraic stratifications}).
By applying a topological version of Kleiman's transversality theorem (see e.g. \cite[p. 39, Theorem 1.3.6 and Example 1.3.7]{gmsmt}), we find a generic linear subspace $H \subset \mathbb{P}^{m}$ of complex codimension $d$ that is transverse to all strata of $X$.
Then, the intersection $Y := H \cap X$ is a pure $0$-dimensional closed subvariety of a smooth Zariski open subvariety $U \subset X$.
Therefore, the closed regular embedding $g \colon Y \hookrightarrow X$ is tight (see \Cref{definition tight embedding}).
Moreover, the inclusion $E \subset \operatorname{Bl}_{Y \times 0}(X \times \mathbb{C})$ of the smooth exceptional divisor $E$ factorizes over the smooth variety $\operatorname{Bl}_{Y \times 0}(U \times \mathbb{C})$, and is therefore topologically normally nonsingular.
Consequently, the embedding $g \colon Y \hookrightarrow X$ is upwardly normally nonsingular (see \Cref{definition upwardly normally nonsingular}).
Thus, \Cref{theorem Verdier-Riemann-Roch type formula} states that the algebraic Gysin map $g^{!}_{\alg} \colon H_{2 \ast}(X; \mathbb{Q}) \rightarrow H_{2\ast-2d}(Y; \mathbb{Q})$ associated to the upwardly normally nonsingular embedding $g \colon Y \hookrightarrow X$ with underlying topological normal bundle $\nu$ satisfies
$$
g^{!}_{\alg}IT_{1, \ast}(X) = L^{\ast}(\nu) \cap IT_{1, \ast}(Y).
$$
Since the compact subvariety $Y \subset \mathbb{P}^{m}$ has pure dimension $0$, it consists of a finite number $k > 0$ of points.
(By construction, $k$ is the degree of the embedding $X \hookrightarrow \mathbb{P}^{m}$, and is hence positive.)
As any vector bundle over a one-point space is trivial, we have  $L^{\ast}(\nu) = 1$.
Thus, $g^{!}_{\alg}IT_{1, \ast}(X) = IT_{1, \ast}(Y)$.
Furthermore, we have $g^{!}_{\alg}[X]_{X} = [Y]_{Y}$ (see e.g. \cite[p. 100, Example 6.2.1]{fultonintth}, and use that the algebraic Gysin map of a closed regular embedding commutes with the cycle map from Chow to Borel-Moore homology according to Verdier \cite[p. 222, 9.2.1]{verdierintcompl}).
Altogether, we conclude that
$$
r \cdot [Y]_{Y} = g^{!}_{\alg}(r \cdot [X]_{X}) = g^{!}_{\alg}IT_{1, 2d}(X) = IT_{1, 0}(Y) \in H_{0}(Y; \mathbb{Q}).
$$
By applying the augmentation $\varepsilon_{\ast} \colon H_{0}(Y; \mathbb{Q}) \rightarrow \mathbb{Q}$ and using that $Y$ is smooth, we obtain
$$
r \cdot k = r \cdot \varepsilon_{\ast}\cdot [Y]_{Y} = \varepsilon_{\ast} IT_{1, 0}(Y) = \varepsilon_{\ast} T_{1, 0}(Y) = \varepsilon_{\ast} L_{ 0}(Y) = \sigma(Y) = k.
$$
Since $k > 0$, we conclude that $r = 1$.
\end{proof}

\begin{remark}\label{remark normalization for it0}
The statement and proof of \Cref{prop top component of it1} hold similarly for $IT_{0, \ast}(X)$.
In fact, \Cref{theorem Verdier-Riemann-Roch type formula} is also valid for $IT_{0, \ast}$ with correction factor $T_{0}^{\ast}(N_{M}W) = \td^{\ast}(N_{M}W)$ instead of $T_{1}^{\ast}(\nu_f) = L^{\ast}(\nu_f)$ as an inspection of the proof of Theorem 6.30 in \cite{banagllgysin} shows.
(In fact, the assumption $y = 1$ is only used in the proof of Proposition 6.26 in \cite{banagllgysin} to conclude that $T_{1}^{\ast}(1_{X}) = L^{\ast}(1_{X}) = 1 \in H^{\ast}(X; \mathbb{Q})$ for the trivial line bundle $1_{X}$ on a complex algebraic variety $X$, but we also have $T_{0}^{\ast}(1_{X}) = \td^{\ast}(1_{X}) = 1 \in H^{\ast}(X; \mathbb{Q})$.)
\end{remark}

\begin{remark}\label{remark alternative proof fundamental class}
For an alternative proof of \Cref{prop top component of it1}, we note that $IT_{y, 2d}(X) = IT_{y, 2d}(X_{\operatorname{reg}})$ under the identification $H_{2d}(X;\mathbb{Q})=H^{\BM}_{2d}(X_{\operatorname{reg}}; \mathbb{Q})$, where $X_{\operatorname{reg}}$ denotes the open smooth part of $X$.
For the manifold $M = X_{\operatorname{reg}}$, we observe $IT_{y, \ast}(M)=T_{y}^{\ast}(TM) \cap [M]$ with $T_{y}^{0}(TM)=1$.
\end{remark}

In \Cref{proposition it class is l type characteristic class} below, let $\mathcal{X}_{CM}$ denote the family of all inclusions $i \colon X \rightarrow W$ of compact irreducible subvarieties $X$ in smooth varieties $W$ such that $X$ is Cohen-Macaulay.
We note that $\mathcal{X}_{CM}$ has all the required properties.
In fact, all inclusions of Schubert subvarieties in Grassmannians are contained in $\mathcal{X}_{CM}$ because Schubert varieties are Cohen-Macaulay.
Furthermore, the family $\mathcal{X}_{CM}$ is stable under products (because the product of two Cohen-Macaulay schemes is again Cohen-Macaulay, see e.g. the proof of the Lemma in \cite[p. 108]{fp}), isomorphisms of smooth ambient spaces, and inclusions into larger smooth ambient spaces.

\begin{thm}\label{proposition it class is l type characteristic class}
The pair $\mathcal{L} = (\mathcal{L}^{\ast}, \mathcal{L}_{\ast})$ defined by $\mathcal{L}^{\ast}(f) = L^{\ast}(\nu_{f})$ for every inclusion $f \colon M \rightarrow W$ of a smooth closed subvariety $M \subset W$ in a smooth variety $W$ with normal bundle $\nu_{f}$, and by $\mathcal{L}_{\ast}(i) = i_{\ast} IT_{1, \ast}(X)$ for every inclusion $i \colon X \rightarrow W$ of a compact possibly singular subvariety $X \subset W$ in a smooth variety $W$ is a Gysin coherent characteristic class with respect to $\mathcal{X}_{CM}$.
\end{thm}

\begin{proof}
By the properties of the cohomological Hirzebruch $L$-class, the class $\mathcal{L}^{\ast}(f) = L^{\ast}(\nu_{f}) \in H^{\ast}(M; \mathbb{Q})$ is normalized for all $f$.
Moreover, by \Cref{prop top component of it1}, the highest nontrivial homogeneous component of the class $IT_{1, \ast}(X) = IT_{1, 0}(X) + IT_{1, 1}(X) + \dots \in H_{\ast}(X; \mathbb{Q})$ is $IT_{1, 2d}(X) = [X]_{X}$, where $d$ denotes the complex dimension of $X$, and $X$ may be assumed to be irreducible by \Cref{remark top class only for irreducible varieties}.
Consequently, the highest nontrivial homogeneous component of $\mathcal{L}_{\ast}(i) = i_{\ast} IT_{1, \ast}(X) \in H_{\ast}(W; \mathbb{Q})$ is the ambient fundamental class $i_{\ast}IT_{1, 2d}(X) = i_{\ast}[X]_{X} = [X]_{W}$.

We proceed to check the axioms of Gysin coherent characteristic classes for the pair $\mathcal{L}$.
As for axiom (\ref{axiom multiplicativity}), we have $IT_{1, \ast}(X \times X') = IT_{1, \ast}(X) \times IT_{1, \ast}(X')$ in $H_{\ast}(X \times X'; \mathbb{Q})$ for all pure-dimensional compact complex algebraic varieties $X$ and $X'$ by \Cref{corollary multiplicativity of intersection generalized todd class}.
Hence, for every $i \colon X \rightarrow W$ and $i' \colon X' \rightarrow W'$, the claim follows by applying $(i \times i')_{\ast}$ and using naturality of the cross product:
\begin{align*}
\mathcal{L}_{\ast}(i \times i') &= (i \times i')_{\ast}IT_{1, \ast}(X \times X') = (i \times i')_{\ast}(IT_{1, \ast}(X) \times IT_{1, \ast}(X')) \\
&= i_{\ast}IT_{1, \ast}(X) \times i_{\ast}'IT_{1, \ast}(X') = \mathcal{L}_{\ast}(i) \times \mathcal{L}_{\ast}(i').
\end{align*}
Next, let us show that the pair $\mathcal{L}$ is compatible with ambient isomorphisms as stated in axiom (\ref{axiom isomorphism}).
As for $\mathcal{L}^{\ast}$, we consider $f \colon M \rightarrow W$ and $f' \colon M' \rightarrow W'$, and an isomorphism $W \stackrel{\cong}{\longrightarrow} W'$ that restricts to an isomorphism $\phi \colon M \stackrel{\cong}{\longrightarrow} M'$.
Then, we have $\phi^{\ast}\nu_{f'} = \nu_{f}$, and thus
$$
\phi^{\ast}\mathcal{L}^{\ast}(f') = \phi^{\ast}L^{\ast}(\nu_{f'}) = L^{\ast}(\phi^{\ast}\nu_{f'}) = L^{\ast}(\nu_{f}) = \mathcal{L}^{\ast}(f).
$$
As for $\mathcal{L}_{\ast}$, we consider $i \colon X \rightarrow W$ and $i' \colon X' \rightarrow W'$, and an isomorphism $\Phi \colon W \stackrel{\cong}{\longrightarrow} W'$ that restricts to an isomorphism $\Phi_{0} \colon X \stackrel{\cong}{\longrightarrow} X'$.
Then, we have $\Phi_{0 \ast}IT_{1, \ast}(X) = IT_{1, \ast}(X')$ by \Cref{prop isomorphism invariance of of intersection generalized todd class}.
Hence, we obtain
$$
\Phi_{\ast}\mathcal{L}_{\ast}(i) = \Phi_{\ast}i_{\ast}IT_{1, \ast}(X) = i_{\ast}'\Phi_{0\ast}IT_{1, \ast}(X) = i_{\ast}'IT_{1, \ast}(X') = \mathcal{L}_{\ast}(i').
$$
To verify axiom (\ref{axiom locality}), we consider $i \colon X \rightarrow W$ and $f \colon M \rightarrow W$ such that $X \subset M$.
Then, the inclusion $i^{M} := i| \colon X \rightarrow M$ satisfies $f \circ i^{M} = i$, and we obtain
$$
f_{\ast}\mathcal{L}_{\ast}(i^{M}) = f_{\ast}i^{M}_{\ast}IT_{1, \ast}(X) = i_{\ast}IT_{1, \ast}(X) = \mathcal{L}_{\ast}(i).
$$
Finally, to show axiom (\ref{axiom gysin}), let us call closed irreducible subvarieties $Z, Z' \subset W$ of a smooth variety $W$ \emph{$\mathcal{X}_{CM}$-transverse} if $Z$ and $Z'$ are simultaneously complex algebraic Whitney transverse (that is, they admit complex algebraic Whitney stratifications such that every stratum of $Z$ is transverse to every stratum of $Z'$ as smooth submanifolds of $W$), generically transverse (see e.g. \cite[Section 3]{bw}), and Tor-independent (see \Cref{definition tor independence}) in $W$.
This notion of $\mathcal{X}_{CM}$-transversality has indeed all required properties.
(In fact, properness of $\mathcal{X}_{CM}$-transverse intersections follows from generic transversality according to \cite[Corollary 3.4]{bw}.
Next, to obtain the desired analog of Kleiman's transversality theorem for the action of $GL_{n}(\mathbb{C})$ on the Grassmannians $G = G_{k}(\mathbb{C}^{n})$, we consider inclusions $X \rightarrow G$ and $X' \rightarrow G$ in $\mathcal{X}_{CM}$.
Then, we apply suitable versions of Kleiman's transversality theorem to obtain an open dense subset $U$ of $GL_{n}(\mathbb{C})$ (in the complex topology) such that $X$ is $\mathcal{X}_{CM}$-transverse to $g \cdot X'$ for all $g \in U$.
Here, Kleiman's transversality theorem holds for (complex algebraic) Whitney transversality by \cite[Theorem 2.2]{bw} since $X$ and $X'$ are compact, for generic transversality by \cite[Theorem 3.5]{bw} since $X$ and $X'$ are irreducible, and for Tor-independence by Sierra's general homological Kleiman-Bertini theorem \cite[Corollary 4.3, p. 608]{sierra} since $X$ and $X'$ are Cohen-Macaulay.
We also note that Zariski dense open subsets are also dense in the complex topology by \cite[Theorem 1, p. 58]{mumfordred}.
Finally, locality holds evidently for our notion of $\mathcal{X}_{CM}$-transversality.)
Now, consider an inclusion $i \colon X \rightarrow W$ in $\mathcal{X}_{CM}$ and an inclusion $f \colon M \rightarrow W$ (of a smooth closed subvariety $M \subset W$ in a smooth variety $W$) such that $M$ is irreducible, and $M$ and $X$ are $\mathcal{X}_{CM}$-transverse in $W$.
Then, \Cref{theorem gysin restriction of hodge theoretic cc} implies that the embedding $g \colon Y \hookrightarrow X$ of the compact subvariety $Y = X \cap M \subset X$ is tight, and the algebraic normal bundle $N = N_{Y}X$ of $g$ and the topological normal bundle $\nu$ of the topologically normally nonsingular inclusion underlying $g$ satisfy
$$
g^{!}_{\alg}IT_{1, \ast}(X) = L^{\ast}(N) \cap IT_{1, \ast}(Y) = L^{\ast}(\nu) \cap IT_{1, \ast}(Y),
$$
where $g^{!}_{\alg} \colon H_{\ast}(X; \mathbb{Q}) \rightarrow H_{\ast}(Y; \mathbb{Q})$ denotes the algebraic Gysin homomorphism associated to $g$ as constructed in Verdier \cite{verdierintcompl} (where we may use the natural identification of Borel-Moore homology and singular homology because $X$ and $Y$ are compact).
Furthermore, since $i \in \mathcal{X}_{CM}$ and $M$ and $X$ are $\mathcal{X}_{CM}$-transverse in $W$, \Cref{alg and top gysin} and \Cref{top gysin borel moore and singular coincide} imply that the algebraic Gysin map $g^{!}_{\alg} \colon H_{\ast}(X; \mathbb{Q}) \rightarrow H_{\ast}(Y; \mathbb{Q})$ coincides with the topological Gysin map $g^{!}_{\operatorname{top}} \colon H_{\ast}(X; \mathbb{Q}) \rightarrow H_{\ast}(Y; \mathbb{Q})$ on all fundamental classes $[Z]_{X}$ of closed irreducible subvarieties $Z \subset X$.
As $IT_{1, \ast}(X)$ is an algebraic cycle according to \Cref{rem.itisalgebraic}, we obtain
$$
g^{!}_{\operatorname{top}}IT_{1, \ast}(X) = g^{!}_{\alg}IT_{1, \ast}(X).
$$
Next, recall from \Cref{normally nonsingular inclusions induced by transverse intersections} that the inclusion $g \colon Y \hookrightarrow X$ is normally nonsingular with topological normal bundle $\nu = j^{\ast} \nu_{f}$ given by the restriction under the inclusion $j \colon Y \rightarrow M$ of the normal bundle $\nu_{f}$ of $M$ in $W$.
Using the base change $f^{!}_{\operatorname{top}}i_{\ast} = j_{\ast}g^{!}_{\operatorname{top}}$ for topological Gysin maps (see \cite[Proposition 2.4]{bw}), as well as $L^{\ast}(\nu) = L^{\ast}(j^{\ast} \nu_{f}) = j^{\ast} L^{\ast}(\nu_{f})$, we conclude that
\begin{align*}
f^{!}_{\operatorname{top}} \mathcal{L}_{\ast}(i) &= f^{!}_{\operatorname{top}} i_{\ast}IT_{1, \ast}(X) = j_{\ast}g^{!}_{\operatorname{top}} IT_{1, \ast}(X) = j_{\ast}g^{!}_{\alg} IT_{1, \ast}(X) = j_{\ast}(L^{\ast}(\nu) \cap IT_{1, \ast}(Y)) \\
&= j_{\ast}(j^{\ast} L^{\ast}(\nu_{f}) \cap IT_{1, \ast}(Y)) = L^{\ast}(\nu_{f}) \cap j_{\ast}IT_{1, \ast}(Y) = \mathcal{L}^{\ast}(f) \cap \mathcal{L}_{\ast}(j).
\end{align*}
\end{proof}

\begin{remark}\label{remark analogy for it0}
A similar proof shows that the pair $(c\ell^{\ast}, c\ell_{\ast})$ given by $c\ell^{\ast}(f \colon M \hookrightarrow W) = \td^*(N_{M}W)$ and $c\ell_{\ast}(i \colon X \hookrightarrow W) = i_{\ast}IT_{0,\ast}(X)$ is also an example of a Gysin coherent characteristic class with respect to $\mathcal{X}_{CM}$.
In fact, \Cref{theorem Verdier-Riemann-Roch type formula} (and consequently, \Cref{theorem gysin restriction of hodge theoretic cc}) is also valid for $IT_{0, \ast}$ with correction factor $T_{0}^{\ast}(N_{M}W) = \td^{\ast}(N_{M}W)$ instead of $T_{1}^{\ast}(\nu_f) = L^{\ast}(\nu_f)$ as an inspection of the proof of Theorem 6.30 in \cite{banagllgysin} shows (compare \Cref{remark normalization for it0}).
See \Cref{remark analogy for ichern} for $IT_{-1,\ast}(X)$.

When $X$ is singular, the intersection Todd class $IT_{0 \ast}(X)$ is generally different from the Baum-Fulton-MacPherson Todd class $\td_{\ast}(X)$ studied in \Cref{example todd class}.
Let us consider the normalization $\pi \colon X' \rightarrow X$ of a singular curve $X$ such that $\pi$ is a homeomorphism.
Let $T_{0 \ast} = \tau_{\ast} \circ mC_{0}$, where $mC_{0} \colon K_{0}(\operatorname{var}/X) \rightarrow K_{0}^{\alg}(X)$ is the evaluation at $y=0$ of the motivic Chern class transformation $mC_{y}$ of \cite{bsy}, recalled here in \Cref{remark motivic chern class trafo}.
Following \cite[Example 3.1]{bsy}, let us show that $T_{0 \ast}(X) \neq \td_{\ast}(X)$.
The normalization $\pi$ restricts to an isomorphism of regular parts.
It is also an isomorphism of the singular parts since those are finite sets of points and the homeomorphism $\pi$ restricts to a bijection between them.
Thus, the scissor relation in the motivic group $K_{0}(\operatorname{var}/X)$ yields $\pi_{\ast}[\operatorname{id}_{X'}] = [\operatorname{id}_{X}]$.
Then, as the smoothness of $X'$ implies $mC_{0} [\operatorname{id}_{X'}] = [\mathcal{O}_{X'}]$ (Du Bois would suffice), we find
\begin{align*}
T_{0 \ast}(X) &:= T_{0 \ast}[\operatorname{id}_{X}] = T_{0 \ast}\pi_{\ast}[\operatorname{id}_{X'}] = \pi_{\ast}T_{0\ast} [\operatorname{id}_{X'}] = \pi_{\ast}\tau_{\ast} mC_{0} [\operatorname{id}_{X'}] = \pi_{\ast} \tau_{\ast} [\mathcal{O}_{X'}].
\end{align*}
As $X$ is singular,
$$
\pi_{\ast}[\mathcal{O}_{X'}] = [\mathcal{O}_{X}] + n \cdot [\mathcal{O}_{\operatorname{pt}}], \qquad n > 0.
$$
Hence,
\begin{align*}
T_{0 \ast}(X) & = \pi_{\ast} \tau_{\ast} [\mathcal{O}_{X'}] =  \tau_{\ast}\pi_{\ast} [\mathcal{O}_{X'}] = \tau_{\ast}([\mathcal{O}_{X}] + n \cdot [\mathcal{O}_{\operatorname{pt}}]) = \td_{\ast}(X) + n \cdot [\operatorname{pt}] \neq \td_{\ast}(X).
\end{align*}
On the other hand, since $\pi$ is a resolution of singularities, $IC_{X}^{H}$ is a direct summand of $\pi_{\ast}\mathbb{Q}_{X'}^{H}[1]$ (see \cite[Corollary 4.6]{schuermannmsri}).
As $\pi$ is in fact a small resolution, we have $\pi_{\ast}(\mathbb{Q}_{X'}^{H}) = IC_{X}^{H}[-1]$. Moreover, in view of \Cref{remark motivic chern class trafo}, we have $[\mathcal{O}_{X'}] = mC_{0} [\operatorname{id}_{X'}] = MHC_{0\ast} \chi_{Hdg} [\operatorname{id}_{X'}] = MHC_{0\ast}[\mathbb{Q}_{X'}^{H}]$.
All in all, we obtain
\begin{align*}
T_{0 \ast}(X) &= \pi_{\ast}\tau_{\ast} [\mathcal{O}_{X'}] = \pi_{\ast}\tau_{\ast}MHC_{0\ast}[\mathbb{Q}_{X'}^{H}] = \pi_{\ast} MHT_{0 \ast}[\mathbb{Q}_{X'}^{H}] = MHT_{0 \ast}[\pi_{\ast} \mathbb{Q}_{X'}^{H}] \\
&= MHT_{0 \ast}[IC_{X}^{H}[-1]] = IT_{0 \ast}(X).
\end{align*}
Therefore, $IT_{0 \ast}(X) = T_{0 \ast}(X) \neq \td_{\ast}(X)$.
(A similar example shows that $IT_{0 \ast}(X) \neq T_{0 \ast}(X)$ in general.)
\end{remark}

Recall from \Cref{example l class as gysin coherent cc} that the pair $(c\ell^{\ast}, c\ell_{\ast})$ given by $c\ell^{\ast}(f \colon M \hookrightarrow W) = L^{\ast}(\nu_{M \subset W})$ and $c\ell_{\ast}(i \colon X \hookrightarrow W) = i_{\ast}L_{\ast}(X)$, where $L_{\ast}$ is the Goresky-MacPherson $L$-class, forms a Gysin coherent characteristic class, and hence a Gysin coherent characteristic class with respect to $\mathcal{X}_{CM}$ as defined at the beginning of this section.
Since the $L$-genus, i.e. the signature, agrees with the genus of $IT_{1 \ast}$ on complex projective algebraic varieties by Saito's intersection cohomology Hodge index theorem (see \cite{saito88}, \cite[Section 3.6]{mss}), the uniqueness theorem for Gysin coherent characteristic classes (\Cref{main result on Gysin coherent characteristic classes with respect to x}) implies

\begin{thm}\label{main result conjecture}
We have $i_{\ast}L_{\ast}(X) = i_{\ast}IT_{1, \ast}(X)$ for all inclusions $i \colon X \rightarrow G$ of compact irreducible Cohen-Macaulay subvarieties in ambient Grassmannians.
\end{thm}

Since Schubert varieties in a Grassmannian are Cohen-Macaulay, and their homology injects into the homology of the ambient Grassmannian, we obtain

\begin{cor}\label{conjecture for schubert holds}
We have $L_{\ast}(X) = IT_{1, \ast}(X)$ for all Schubert varieties $X$ in a Grassmannian.
\end{cor}

\section{The Chern Class as a Gysin Coherent Characteristic Class}\label{example chern class}
Let $c_{\ast} \colon F(X) \longrightarrow H^\BM_{2*} (X)\otimes \rat$ denote the (rationalized) Chern class transformation of MacPherson \cite{macph} defined on the group $F(X)$ of $\mathbb{Z}$-valued algebraically constructible functions on the possibly singular complex algebraic variety $X$.
(Such a constructible function is a linear combination of indicator functions $1_{Z}$ with $Z \subset X$ a closed irreducible subvariety.)
The Chern-Schwartz-MacPherson class $c_{\ast}(X)$ of a possibly singular irreducible complex algebraic variety $X$ is defined as
$$
c_{\ast}(X) := c_{\ast}(1_{X}) \in H_{2\ast}^{BM}(X; \mathbb{Q}).
$$
While the Chern classes $c_{\ast}(X)$ were originally defined with integral coefficients, we need to consider them with rational coefficients to be able to study them in the framework of Gysin coherence.

In \cite{schuertrans}, the second author derives the following Verdier-Riemann-Roch type theorem for the behavior of the Chow homology Chern class transformation $c_{\ast} \colon F(X) \rightarrow A_{\ast}(X)$ under refined Gysin maps associated to transverse intersections in a microlocal context.

\begin{thm}[see Corollary 2.7 in \cite{schuertrans}]\label{theorem vrr chern}
Let $f \colon M \rightarrow W$ be a closed embedding of smooth complex varieties with algebraic normal bundle $N = N_{M}W$.
Let $X \subset W$ be a (Zariski) closed subspace, and set $Y := f^{-1}(X) = M \cap X \subset M$.
Assume that $\gamma \in F(X)$ is a constructible function such that $f$ is non-characteristic with respect to the support $\operatorname{supp}(CC(\gamma))$ of the characteristic cycle $CC(\gamma)$ of $\gamma$.
Then,
$$
(f, i)^{!}_{\operatorname{ref}}(c_{\ast}(\gamma)) = c^{\ast}(N|_{Y}) \cap c_{\ast}(g^{\ast}(\gamma)) \in A_{\ast}(Y),
$$
where $(f, i)^{!}_{\operatorname{ref}} \colon A_{\ast}(X) \rightarrow A_{\ast}(Y)$ denotes the refined Gysin map associated to the cartesian square of closed embeddings
\[ \xymatrix{
Y := M\cap X \ar[d] \ar^-{g}[r] &
  X \ar[d]^{i} \\
M \ar[r]_f & W
} \]
(see Fulton \cite[Section 6.2]{fultonintth}).
\end{thm}

\begin{remark}\label{remark transversality implies non-characteristic property}
As pointed out in \cite{schuertrans}, the non-characteristic condition for $\gamma \in F(X)$ in \Cref{theorem vrr chern} above is for example satisfied when $M$ is transverse to all strata of a complex algebraic Whitney stratification of $X \subset W$, and $\gamma$ is constructible with respect to this stratification, i.e., $\gamma|_{S}$ is locally constant for all strata $S$ of $X$.
\end{remark}

A simple consequence is that the highest non-trivial homogeneous component of the Chern class is the fundamental class (compare \Cref{prop top component of it1}).

\begin{cor}\label{prop top component of chern}
For an irreducible projective variety $X$ of complex dimension $d$, the highest non-trivial homogeneous component of the class
$$
c_{\ast}(X) = c_{0}(X) + c_{1}(X) + \dots \in A_{\ast}(X)
$$
is the fundamental class, $c_{d}(X) = [X]_{X}$.
\end{cor}

\begin{proof}
Since $X$ is irreducible, we know that $A_{d}(X)$ is generated by the fundamental class $[X]_{X}$, and that $A_{i}(X) = 0$ for $i > d$.
Therefore, writing $c_{d}(X) = r \cdot [X]_{X}$ for some $r \in \mathbb{Z}$, it remains to show that $r = 1$.
For this purpose, we fix an embedding $X \subset \mathbb{P}^{m} =: W$ and a complex algebraic Whitney stratification on $X$ (see \Cref{example complex algebraic Whitney stratifications are algebraic stratifications}).
By applying a topological version of Kleiman's transversality theorem (see e.g. \cite[p. 39, Theorem 1.3.6 and Example 1.3.7]{gmsmt}), we find a generic linear subspace $H \subset \mathbb{P}^{m} = W$ of complex codimension $d$ that is transverse to all strata of $X$.
Then, the intersection $Y := H \cap X$ is a pure $0$-dimensional closed subvariety of a Zariski open subset $U \subset X$.
Hence, the closed embedding $g \colon Y \hookrightarrow X$ is regular since it is the composition of a smooth embedding of smooth varieties and an open embedding.
Therefore, the refined Gysin map $(f, i)^{!}_{\operatorname{ref}} \colon A_{\ast}(X) \rightarrow A_{\ast}(Y)$ associated to the cartesian square of closed embeddings
\[ \xymatrix{
Y := H\cap X \ar[d] \ar[r]^-{g} &
  X \ar[d]^{i} \\
H \ar[r]_f & W
} \]
coincides with the algebraic Gysin map $g^{!}_{\operatorname{alg}} \colon A_{\ast}(X) \rightarrow A_{\ast}(Y)$ because $f$ and $g$ are both codimension $d$ regular embeddings (see Fulton \cite[p. 99, Remark 6.2.1]{fultonintth}).
By \Cref{remark transversality implies non-characteristic property}, the constructible function
$$
\gamma := 1_{X} \in F(X)
$$
satisfies the non-characteristic property of \Cref{theorem vrr chern} because $M := H$ is by assumption transverse to all strata of the given complex algebraic Whitney stratification on $X$, and $\gamma|_{S} = 1_{S}$ is locally constant for all strata $S$ of $X$.
The pull-back $g^{\ast} \colon F(X) \rightarrow F(Y)$ induced by the closed embedding $g \colon Y \hookrightarrow X$ satisfies $g^{\ast}(\gamma) = g^{\ast}(1_{X}) = 1_{H \cap X} = 1_{Y}$.
Hence, \Cref{theorem vrr chern} yields
\begin{equation*}
g^{!}_{\operatorname{alg}}(c_{\ast}(X)) = (f, i)^{!}_{\operatorname{ref}}(c_{\ast}(X)) = c^{\ast}(N|_{Y}) \cap c_{\ast}(Y) \in A_{\ast}(Y).
\end{equation*}
Since the compact subvariety $Y \subset \mathbb{P}^{m}$ has pure dimension $0$, it consists of a finite number $k > 0$ of points.
(By construction, $k$ is the degree of the embedding $i \colon X \hookrightarrow \mathbb{P}^{m}$, and is hence positive.)
As any vector bundle over a one-point space is trivial, we have  $c^{\ast}(N|_{Y}) = 1$.
Thus, $g^{!}_{\operatorname{alg}}(c_{\ast}(X)) = c_{\ast}(Y)$.
Furthermore, we have $g^{!}_{\operatorname{alg}}([X]_{X}) = [Y]_{Y}$ (see e.g. \cite[p. 100, Example 6.2.1]{fultonintth}).
Altogether, we conclude that
$$
r \cdot [Y]_{Y} = (f, i)^{!}_{\operatorname{ref}}(r \cdot [X]_{X}) = (f, i)^{!}_{\operatorname{ref}}(c_{d}(X)) = c_{0}(Y) \in A_{0}(Y).
$$
By applying the augmentation $\varepsilon_{\ast} \colon A_{0}(Y) \rightarrow \mathbb{Z}$ and using that $Y$ is smooth, we obtain
$$
r \cdot k = r \cdot \varepsilon_{\ast}\cdot [Y]_{Y} = \varepsilon_{\ast} c_{0}(Y) = k.
$$
Since $k > 0$, we conclude that $r = 1$.
\end{proof}

\begin{remark}
As in \Cref{remark alternative proof fundamental class}, an alternative proof of \Cref{prop top component of chern} can be obtained by observing that the restriction of $c_{d}(X)$ to the regular part of $X$ is $c_{d}(X_{\operatorname{reg}})$.
The restriction $A_{d}(X) \rightarrow A_{d}(X_{\operatorname{reg}})$ is an isomorphism by \cite[Proposition 1.8, p. 21]{fultonintth} (recall that $X$ is assumed to be irreducible).
\end{remark}

\Cref{theorem vrr chern} implies the following result.

\begin{cor}\label{thm algebraic Gysin homomorphism fundamental classes without tor independence}
Let $X \hookrightarrow W \hookleftarrow M$ be closed algebraic embeddings of pure-dimensional complex quasiprojective algebraic varieties with $M, W$ smooth.
Let $V \subset X$ be an irreducible closed subvariety.
We suppose that $X \subset W$ is equipped with a complex algebraic Whitney stratification such that $M$ is transverse to all strata, and $V$ is a union of strata.
Then the refined Gysin map
\[ (f, i)^{!}_{\operatorname{ref}} \colon A_{\ast} (X) \longrightarrow A_{\ast-c} (Y) \]
associated to the cartesian square of inclusions
\[ \xymatrix{
Y := M\cap X \ar[d] \ar[r]^-{g} &
  X \ar[d]^{i} \\
M \ar[r]_f & W
} \]
satisfies $(f, i)^{!}_{\operatorname{ref}}[V]_{X} = [V \cap Y]_{Y}$, where $c$ denotes the complex codimension of $M$ in $W$.
\end{cor}

\begin{proof}
By \Cref{remark transversality implies non-characteristic property}, the constructible function
$$
\gamma := 1_{V} \in F(X)
$$
satisfies the non-characteristic property of \Cref{theorem vrr chern} because the given complex algebraic Whitney stratification on $X$ is such that $M$ is transverse to all strata, and $V$ is a union of strata.
The pull-back $g^{\ast} \colon F(X) \rightarrow F(Y)$ satisfies $g^{\ast}(\gamma) = g^{\ast}(1_{V}) = 1_{V \cap Y}$.
Hence, \Cref{theorem vrr chern} yields
\begin{equation}\label{equation vrr chern for subvariety}
(f, i)^{!}_{\operatorname{ref}}(c_{\ast}(1_{V})) = c^{\ast}(N|_{Y}) \cap c_{\ast}(1_{V \cap Y}) \in A_{\ast}(Y).
\end{equation}
The inclusions $\alpha \colon V \hookrightarrow X$ and $\beta \colon V \cap Y \hookrightarrow Y$ are closed embeddings, and hence proper.
The induced maps $\alpha_{\ast} \colon F(V) \rightarrow F(X)$ and $\beta_{\ast} \colon F(V \cap Y) \rightarrow F(Y)$ satisfy $\alpha_{\ast}(1_{V}) = 1_{V}$ and $\beta_{\ast}(1_{V \cap Y}) = 1_{V \cap Y}$, respectively.
(In general, the push-forward $\varphi_{\ast} \colon F(U) \rightarrow F(U')$ of a proper morphism $\varphi \colon U \rightarrow U'$ is defined on a subvariety $Z \subset U$ by $\varphi_{\ast}(1_{Z})(v) = \chi(Z \cap \varphi^{-1}(v))$, where $\chi$ denotes the topological Euler characteristic, see \cite[p. 376, Example 19.1.7]{fultonintth}.)
Since $c_{\ast}$ commutes with proper push-forward according to \cite[p. 377, Example 19.1.7]{fultonintth}, equation (\ref{equation vrr chern for subvariety}) becomes
\begin{equation}\label{equation vrr chern for subvariety 2}
(f, i)^{!}_{\operatorname{ref}}(\alpha_{\ast}c_{\ast}(V)) = c^{\ast}(N|_{Y}) \cap \beta_{\ast}c_{\ast}(V \cap Y) \in A_{\ast}(Y).
\end{equation}
Let $d$ denote the complex dimension of $V$.
Then, the highest non-trivial homogeneous component of $c_{\ast}(V)$ is $c_{d}(V) = [V]_{V}$ by \Cref{prop top component of chern}.
Similarly, since $V \cap Y = V \cap M$ is pure $(d-c)$-dimensional, the highest non-trivial homogeneous component of $c_{\ast}(V \cap Y)$ is $c_{d-c}(V \cap Y) = [V \cap Y]_{V \cap Y}$.
Consequently, evaluation of equation (\ref{equation vrr chern for subvariety 2}) in degree $d-c$ yields
$$
(f, i)^{!}_{\operatorname{ref}}([V]_{X}) = [V \cap Y]_{Y} \in A_{d-c}(Y).
$$
\end{proof}

Since $c_{\ast}(X) \in A_{\ast}(X)$ is an algebraic cycle, there exist a finite number of irreducible closed subvarieties $V_{1}, \dots, V_{r} \subset X$ such that
\begin{equation}\label{chern is algebraic cycle}
c_{\ast}(X) = \sum_{l=1}^{r} \lambda_{l}[V_{l}]_{X}, \quad \lambda_{l} \in \mathbb{Z}.
\end{equation}
For an irreducible closed subvariety $X \subset W$ of a smooth variety $W$, we call a complex algebraic Whitney stratification of $X \subset W$ \emph{$c_{\ast}(X)$-constructible} if there is a representation (\ref{chern is algebraic cycle}) in which every $V_{l}$ is a union of strata of $X$.
It follows from \cite[p. 43, Section 1.7, Theorem]{gmsmt} that such Whitney stratifications exist on $X \subset W$.

\begin{thm}\label{thm chern satisfies topological gysin coherence}
Let $X \hookrightarrow W \hookleftarrow M$ be closed algebraic embeddings of pure-dimensional complex algebraic varieties with $M, W$ smooth.
Suppose that $X$ is equipped with a complex algebraic Whitney stratification that is $c_{\ast}(X)$-constructible, and such that $M$ is transverse to all strata of $X$ and $Y = M \cap X$ is compact.
Then, the topologically normally nonsingular inclusion $g \colon Y = M \cap X \hookrightarrow X$ satisfies
$$
g^{!!}_{\operatorname{top}}(c_{\ast}(X)) = c^{\ast}(N|_{Y}) \cap c_{\ast}(Y) \in H^\BM_{2*} (Y) \otimes \mathbb{Q},
$$
where $N = N_{M}W$ denotes the complex normal bundle of $M$ in $W$.
\end{thm}

\begin{proof}
By \Cref{remark transversality implies non-characteristic property}, the constructible function
$$
\gamma := 1_{X} \in F(X)
$$
satisfies the non-characteristic property of \Cref{theorem vrr chern} because $M$ is by assumption transverse to all strata of the given complex algebraic Whitney stratification on $X$, and $\gamma|_{S} = 1_{S}$ is locally constant for all strata $S$ of $X$.
The pull-back $g^{\ast} \colon F(X) \rightarrow F(Y)$ induced by the closed embedding $g \colon Y \hookrightarrow X$ satisfies $g^{\ast}(\gamma) = g^{\ast}(1_{X}) = 1_{M \cap X} = 1_{Y}$.
Hence, \Cref{theorem vrr chern} yields
\begin{equation*}
(f, i)^{!}_{\operatorname{ref}}(c_{\ast}(X)) = c^{\ast}(N|_{Y}) \cap c_{\ast}(Y) \in A_{\ast}(Y).
\end{equation*}
By assumption, the Chern-Schwartz-MacPherson class $c_{\ast}(X)$ can be written in the form (\ref{chern is algebraic cycle}) for some irreducible closed subvarieties $V_{1}, \dots, V_{r} \subset X$ such that every $V_{l}$ is a union of strata of the given complex algebraic Whitney stratification on $X$.
Therefore, \Cref{thm algebraic Gysin homomorphism fundamental classes without tor independence} implies that
\begin{equation*}
(f, i)^{!}_{\operatorname{ref}}(c_{\ast}(X)) = (f, i)^{!}_{\operatorname{ref}}\sum_{l=1}^{r} \lambda_{l}[V_{l}]_{X} = \sum_{l=1}^{r} \lambda_{l} \cdot (f, i)^{!}_{\operatorname{ref}}([V_{l}]_{X}) = \sum_{i=1}^{r} \lambda_{l}[V_{l} \cap Y]_{Y}.
\end{equation*}
By applying the cycle map $\cl: A_* (Y) \to H^\BM_{2*} (Y) \otimes \mathbb{Q}$ from Chow homology to Borel-Moore homology, we conclude that
\begin{align*}
\sum_{i=1}^{r} \lambda_{l}[V_{l} \cap Y]_{Y} &= \cl (f, i)^{!}_{\operatorname{ref}}(c_{\ast}(X)) = \cl (c^{\ast}(N|_{Y}) \cap c_{\ast}(Y)) \\
&= c^{\ast}(N|_{Y}) \cap \cl (c_{\ast}(Y)) \in H^\BM_{2*} (Y) \otimes \mathbb{Q},
\end{align*}
where the cycle map $\cl$ and the cap product with Chern classes are compatible by \cite[p. 374, Prop. 19.1.2]{fultonintth}.
Furthermore, by the analog of \Cref{proposition topological Gysin homomorphism fundamental classes} for the topological Gysin map in Borel-Moore homology (see \Cref{remark borel moore gysin fundamental classes}), we have
$$
\sum_{i=1}^{r} \lambda_{l}[V_{l} \cap Y]_{Y} = \sum_{i=1}^{r} \lambda_{l} g^{!!}_{\operatorname{top}}([V_{l}]_{X}) = g^{!!}_{\operatorname{top}}\sum_{i=1}^{r} \lambda_{l}[V_{l}]_{X} = g^{!!}_{\operatorname{top}}(c_{\ast}(X)) \in H^\BM_{2*} (Y) \otimes \mathbb{Q}.
$$
\end{proof}

Note that the following theorem does not require a Cohen-Macaulay assumption.

\begin{thm}\label{thm chern classes are gysin coherent cc}
The pair $(c\ell^{\ast}, c\ell_{\ast})$ given by $c\ell^{\ast}(f \colon M \hookrightarrow W) = c^{\ast}(N_{M}W)$ and $c\ell_{\ast}(i \colon X \hookrightarrow W) = i_{\ast}c_{\ast}(X)$ forms a Gysin coherent characteristic class.
\end{thm}

\begin{proof}
By the properties of the cohomological Chern class $c^{\ast}$, the class $c\ell^{\ast}(f) = c^{\ast}(N_{f}) \in H^{\ast}(M; \mathbb{Q})$ is normalized for all $f$.
Moreover, by \Cref{prop top component of chern}, the highest nontrivial homogeneous component of the class $c_{\ast}(X) = c_{0}(X) + c_{1}(X) + \dots \in H_{\ast}(X; \mathbb{Q})$ is $c_{d}(X) = [X]_{X}$, where $d$ denotes the complex dimension of $X$, and $X$ may be assumed to be irreducible by \Cref{remark top class only for irreducible varieties}.
Consequently, the highest nontrivial homogeneous component of $c\ell_{\ast}(i) = i_{\ast} c_{\ast}(X) \in H_{\ast}(W; \mathbb{Q})$ is the ambient fundamental class $i_{\ast}c_{d}(X) = i_{\ast}[X]_{X} = [X]_{W}$.

We proceed to check the axioms of Gysin coherent characteristic classes for the pair $c\ell$.
As for axiom (\ref{axiom multiplicativity}), multiplicativity $c_{\ast}(X \times X') = c_{\ast}(X) \times c_{\ast}(X')$ holds for all compact irreducible complex algebraic varieties $X$ and $X'$ by Kwiecinski \cite{kwiec} and Kwiecinski-Yokura \cite{ky}.
Hence, for every $i \colon X \rightarrow W$ and $i' \colon X' \rightarrow W'$, the claim follows by applying $(i \times i')_{\ast}$ and using naturality of the cross product:
\begin{align*}
c\ell_{\ast}(i \times i') &= (i \times i')_{\ast}c_{\ast}(X \times X') = (i \times i')_{\ast}(c_{\ast}(X) \times c_{\ast}(X')) \\
&= i_{\ast}c_{\ast}(X) \times i_{\ast}'c_{\ast}(X') = c\ell_{\ast}(i) \times c\ell_{\ast}(i').
\end{align*}
Next, let us show that the pair $c\ell$ is compatible with ambient isomorphisms as stated in axiom (\ref{axiom isomorphism}).
As for $c^{\ast}$, we consider $f \colon M \rightarrow W$ and $f' \colon M' \rightarrow W'$, and an isomorphism $W \stackrel{\cong}{\longrightarrow} W'$ that restricts to an isomorphism $\phi \colon M \stackrel{\cong}{\longrightarrow} M'$.
Then, we have $\phi^{\ast}N_{f'} = N_{f}$, and thus
$$
\phi^{\ast}c\ell^{\ast}(f') = \phi^{\ast}c^{\ast}(N_{f'}) = c^{\ast}(\phi^{\ast}N_{f'}) = c^{\ast}(N_{f}) = c\ell^{\ast}(f).
$$
As for $c\ell_{\ast}$, we consider $i \colon X \rightarrow W$ and $i' \colon X' \rightarrow W'$, and an isomorphism $\Phi \colon W \stackrel{\cong}{\longrightarrow} W'$ that restricts to an isomorphism $\Phi_{0} \colon X \stackrel{\cong}{\longrightarrow} X'$.
Then, we have
$$
c_{\ast}(X') = c_{\ast}(1_{X'}) = c_{\ast}(\Phi_{0 \ast} 1_{X}) = \Phi_{0 \ast} c_{\ast}(1_{X}) = \Phi_{0 \ast} c_{\ast}(X)
$$
because $c_{\ast} \colon F(X) \longrightarrow H^\BM_{2*} (X)\otimes \rat$ commutes with proper push-forward (see e.g. \cite[p. 377, top]{fultonintth}).
Hence, we obtain
$$
\Phi_{\ast}c\ell_{\ast}(i) = \Phi_{\ast}i_{\ast}c_{\ast}(X) = i_{\ast}'\Phi_{0\ast}c_{\ast}(X) = i_{\ast}'c_{\ast}(X') = c\ell_{\ast}(i').
$$
To verify axiom (\ref{axiom locality}), we consider $i \colon X \rightarrow W$ and $f \colon M \rightarrow W$ such that $X \subset M$.
Then, the inclusion $i^{M} := i| \colon X \rightarrow M$ satisfies $f \circ i^{M} = i$, and we obtain
$$
f_{\ast}c\ell_{\ast}(i^{M}) = f_{\ast}i^{M}_{\ast}c_{\ast}(X) = i_{\ast}c_{\ast}(X) = c\ell_{\ast}(i).
$$
Finally, to show axiom (\ref{axiom gysin}), let us recall from \Cref{example typical choice for family x and transversality} that $\mathcal{X} = \mathcal{X}_{0}$ is the family of all inclusions $i \colon X \hookrightarrow W$ of compact irreducible subvarieties $X$ in smooth varieties $W$.
Furthermore, let us call two closed irreducible subvarieties $Z, Z' \subset W$ of a smooth variety $W$ \emph{$\mathcal{X}_{0}$-transverse} if they admit complex algebraic Whitney stratifications such that every stratum of $Z$ is transverse to every stratum of $Z'$, where the stratification of $Z$ (resp. $Z'$) can be chosen to be $c_{\ast}(Z)$-constructible when $Z$ is compact (resp. $c_{\ast}(Z')$-constructible when $Z'$ is compact).
Then, the notion of $\mathcal{X}_{0}$-transversality satisfies all required properties.
(In fact, the intersection $Z \cap Z'$ of two $\mathcal{X}_{0}$-transverse closed irreducible subvarieties $Z, Z' \subset W$ of a smooth variety $W$ is proper, which follows from the transversality of the strata of the complex algebraic Whitney stratifications of $Z$ and $Z'$.
Moreover, for any inclusions $i \colon X \rightarrow G$ and $i' \colon X' \rightarrow G$ in $\mathcal{X}_{0}$, there is a nonempty open dense subset $U \subset GL_{n}(\mathbb{C})$ (in the complex topology) such that $X$ is $\mathcal{X}_{0}$-transverse to $g \cdot X'$ for all $g \in U$.
This can be achieved by applying a topological version of Kleiman's transversality theorem (see e.g. \cite[p. 39, Theorem 1.3.6 and Example 1.3.7]{gmsmt}) to complex algebraic Whitney stratifications on the compact varieties $X$ and $X'$ that are $c_{\ast}(X)$-constructible and $c_{\ast}(X')$-constructible, respectively.
Finally, to see that our notion of $\mathcal{X}_{0}$-transversality satisfies locality, we suppose that $Z, Z' \subset W$ are $\mathcal{X}_{0}$-transverse.
Then, the desired transverse Whitney stratifications on $Z \cap U$ and $Z' \cap U$ are obtained by intersecting those on $Z$ and $Z'$ with $U$.
Moreover, if, say, $Z \cap U$ is compact, and hence a complete variety, then it is also a closed subset of $Z$, see e.g. \cite[p. 55, property i)]{mumfordred}.
But since $Z \cap U$ is also non-empty open subset of the irreducible space $Z$, we conclude that $Z \cap U = Z$.
Hence, $Z$ is compact and $c_{\ast}(Z)$-constructible, that is, $Z \cap U$ is $c_{\ast}(Z \cap U)$-constructible.)
Now, consider an inclusion $i \colon X \rightarrow W$ in $\mathcal{X}_{0}$ and an inclusion $f \colon M \rightarrow W$ (of a smooth closed subvariety $M \subset W$ in a smooth variety $W$) such that $M$ is irreducible, and $M$ and $X$ are $\mathcal{X}_{0}$-transverse in $W$.
By definition of $\mathcal{X}_{0}$-transversality, $X$ can be equipped with a complex algebraic Whitney stratification that is $c_{\ast}(X)$-constructible, and such that $M$ is transverse to all strata of $X$.
Hence, by \Cref{thm chern satisfies topological gysin coherence} and \Cref{top gysin borel moore and singular coincide} (where we may use the natural identification of Borel-Moore homology and singular homology because $X$ and $Y$ are compact), the topologically normally nonsingular inclusion $g \colon Y = M \cap X \hookrightarrow X$ satisfies
$$
g^{!}_{\operatorname{top}}(c_{\ast}(X)) = c^{\ast}(j^{\ast}N) \cap c_{\ast}(Y) \in H_{2*} (Y) \otimes \mathbb{Q},
$$
where $N = N_{M}W$ denotes the complex normal bundle of $M$ in $W$, and $j \colon Y \hookrightarrow M$ is the inclusion.
Using the base change $f^{!}_{\operatorname{top}}i_{\ast} = j_{\ast}g^{!}_{\operatorname{top}}$ for topological Gysin maps (see \cite[Proposition 2.4]{bw}), as well as $c^{\ast}(j^{\ast}N) = j^{\ast} c^{\ast}(N)$, we conclude that
\begin{align*}
f^{!}_{\operatorname{top}} c\ell_{\ast}(i) &= f^{!}_{\operatorname{top}} i_{\ast}c_{\ast}(X) = j_{\ast}g^{!}_{\operatorname{top}} c_{\ast}(X) = j_{\ast}(c^{\ast}(j^{\ast}N) \cap c_{\ast}(Y)) \\
&= j_{\ast}(j^{\ast} c^{\ast}(N) \cap c_{\ast}(Y)) = c^{\ast}(N) \cap j_{\ast}c_{\ast}(Y) = c\ell^{\ast}(f) \cap c\ell_{\ast}(j).
\end{align*}

This completes the proof of \Cref{thm chern classes are gysin coherent cc}.
\end{proof}

\begin{remark}\label{remark analogy for ichern}
A similar proof shows that the pair $(c\ell^{\ast}, c\ell_{\ast})$ given by $c\ell^{\ast}(f) = c^{\ast}(N_{f})$ and $c\ell_{\ast}(i) = i_{\ast}Ic_{\ast}(X)$ is also an example of a Gysin coherent characteristic class.
Here, $Ic_{\ast}(X):=c_{\ast}\chi_{\operatorname{stalk}}(IC_{X}^{H}[-n])$, with the constructible function $\chi_{\operatorname{stalk}}$ given by the stalkwise Euler characteristic.
(For completeness, we also point out that $Ic_{\ast}(X) = IT_{-1, \ast}(X)$, where the specialization $y = -1$ is possible in \Cref{intersection generalized todd class} because we actually have $IT_{y \ast} (X) \in H^\BM_{2 \ast} (X) \otimes \rat [y^{\pm 1}]$ as shown in \cite[p. 465, Proposition 5.21]{schuermannmsri}.
However, an alternative proof in analogy with \Cref{remark analogy for it0} will only yield a Gysin coherent characteristic class with respect to $\mathcal{X}_{CM}$.)
\end{remark}

\begin{remark}\label{remark analogy for chern mather}
The same type of argument can be used to show that the pair $(c\ell^{\ast}, c\ell_{\ast})$ given by $c\ell^{\ast}(f) = c^{\ast}(N_{f})$ and $c\ell_{\ast}(i) = i_{\ast}c_{\ast}^{M}(X)$ is another example of a Gysin coherent characteristic class, where
$$
c_{\ast}^{M}(X) := c_{\ast}(Eu_{X}) \in H_{2\ast}^{BM}(X; \mathbb{Q})
$$
is the Chern-Mather class of a possibly singular irreducible complex projective algebraic variety $X$, and $Eu_{X}$ denotes the Euler obstruction constructible function of $X$ introduced by MacPherson \cite{macph}.
In the proof, one exploits the following well known properties (see e.g. Parusi\'{n}ski and Pragacz \cite[Lemma 1.1]{pp}):
\begin{enumerate}
\item $Eu_X(x)=1$ for $x\in X_{reg}$.
\item $Eu_{X\times X'}(x,x')=E_X(x)\cdot Eu_{X'}(x')$ for $x\in X, x'\in X'$.
\item $Eu_X$ is constructible with respect to any complex Whitney stratification of $X$.
\item $f^*(Eu_X)=Eu_{M\cap X}$ for a complex manifold embedding $f: M\to W$ 
transversal to a complex Whitney stratification of $X$.
\end{enumerate}
\end{remark}

We have seen that the three generalizations $c_{\ast}$, $c_{\ast}^{M}$ and $Ic_{\ast}$ of the Chern class to singular varieties give rise to Gysin coherent classes.
As these classes are generally not equal, their associated genera must already be different in light of the Uniqueness \Cref{main result on Gysin coherent characteristic classes with respect to x} for Gysin coherent classes.
Indeed,
$$
|c_{\ast}|(X) = \chi(X)
$$
the topological Euler characteristic,
$$
|c_{\ast}^{M}|(X) = \chi(X; Eu_{X}) = \sum_{S}Eu_{X}(S) \cdot \chi(S),
$$
where the sum ranges over all connected strata $S$ of a complex algebraic Whitney stratification of $X$, and
$$
|Ic_{\ast}|(X) = |IT_{-1 \ast}|(X) = I \chi_{-1}(X) = \chi(IH^{\ast}(X; \mathbb{Q})),
$$
the intersection cohomology Euler characteristic.

\section{The Todd Class as a Gysin Coherent Characteristic Class}\label{example todd class}
Let $\tau_{\ast} \colon K_0 (X) \longrightarrow H^\BM_{2*} (X)\otimes \rat$ denote the Todd class transformation of Baum, Fulton and MacPherson \cite{bfm}, \cite{fmp}.
Recall from \Cref{rem.bfmtoddtochow} that this transformation is compatible with its Chow homology analog $\tau_{\ast} \colon K_0 (X) \longrightarrow A_{\ast} (X)\otimes \rat$ under the cycle map $\operatorname{cl} \colon A_{\ast} (X)\otimes \rat \rightarrow H^\BM_{2*} (X)\otimes \rat$.
Baum, Fulton and MacPherson \cite{bfm} define the Todd class of a possibly singular complex algebraic variety $X$ as
$$
\td_{\ast}(X) := \tau_{\ast}([\mathcal{O}_{X}]) \in H_{2\ast}^{BM}(X; \mathbb{Q}).
$$
Its genus $|\td_\ast|$ (for $X$ compact) is the arithmetic genus of a singular variety, i.e. the holomorphic Euler characteristic.

Todd classes can be studied within the framework of Gysin coherent characteristic classes by a similar method to that used for the class $IT_{1, \ast}$ in \Cref{the class it}.
In the following result, we are therefore using the same definition of the set $\mathcal{X}_{CM}$ and the notion of $\mathcal{X}_{CM}$-transversality as in \Cref{proposition it class is l type characteristic class} (and its proof) to be able to replace the algebraic by the topological Gysin map in Verdier's Gysin restriction formula for the Todd class.

\begin{thm}\label{proposition todd class is l type characteristic class}
The pair $(c\ell^{\ast}, c\ell_{\ast})$ defined by $c \ell^{\ast}(f) = \operatorname{td}^{\ast}(N_{f})$ for every inclusion $f \colon M \rightarrow W$ of a smooth closed subvariety $M \subset W$ in a smooth variety $W$ with complex normal bundle $N_{f}$, and by $c \ell_{2 \ast}(i) = i_{\ast} \operatorname{td}_{\ast}(X)$ for every inclusion $i \colon X \rightarrow W$ of a compact possibly singular subvariety $X \subset W$ in a smooth variety $W$, is a Gysin coherent characteristic class with respect to $\mathcal{X}_{CM}$.
\end{thm}

\begin{proof}
By the properties of the cohomological Todd class, the class $c \ell^{\ast}(f) = \operatorname{td}^{\ast}(N_{f}) \in H^{\ast}(M; \mathbb{Q})$ is normalized for all $f$.
Moreover, the highest non-vanishing homogeneous component of $\td_{\ast}(X)$ is the fundamental class $\td_{d}(X) = [X]_{X} \in H_{2d}(X; \mathbb{Q})$ according to \cite[p. 353, Theorem 18.3.5(5)]{fultonintth}.
Consequently, the highest nontrivial homogeneous component of $c \ell_{2 \ast}(i) = i_{\ast} \operatorname{td}_{\ast}(X) \in H_{\ast}(W; \mathbb{Q})$ is the ambient fundamental class $i_{\ast}\operatorname{td}_{d}(X) = i_{\ast}[X]_{X} = [X]_{W}$.

We proceed to check the axioms of Gysin coherent characteristic classes for the pair $c \ell$.
As for axiom (\ref{axiom multiplicativity}), the multiplicativity property $\td_{\ast}(X \times X') = \td_{\ast}(X) \times \td_{\ast}(X')$ holds for all compact irreducible complex algebraic varieties $X$ and $X'$ by \cite[p. 360, Example 18.3.1]{fultonintth}.
Hence, for every $i \colon X \rightarrow W$ and $i' \colon X' \rightarrow W'$, the claim follows by applying $(i \times i')_{\ast}$ and using naturality of the cross product:
\begin{align*}
c \ell_{2 \ast}(i \times i') &= (i \times i')_{\ast}\td_{\ast}(X \times X') = (i \times i')_{\ast}(\td_{\ast}(X) \times \td_{\ast}(X')) \\
&= i_{\ast}\td_{\ast}(X) \times i_{\ast}'\td_{\ast}(X') = c \ell_{2\ast}(i) \times c \ell_{2 \ast}(i').
\end{align*}
Next, let us show that the pair $c \ell$ is compatible with ambient isomorphisms as stated in axiom (\ref{axiom isomorphism}).
As for $c \ell^{\ast}$, we consider $f \colon M \rightarrow W$ and $f' \colon M' \rightarrow W'$, and an isomorphism $W \stackrel{\cong}{\longrightarrow} W'$ that restricts to an isomorphism $\phi \colon M \stackrel{\cong}{\longrightarrow} M'$.
Then, we have $\phi^{\ast}N_{f'} = N_{f}$, and thus
$$
\phi^{\ast}c \ell^{\ast}(f') = \phi^{\ast}\td^{\ast}(N_{f'}) = \td^{\ast}(\phi^{\ast}N_{f'}) = \td^{\ast}(N_{f}) = c \ell^{\ast}(f).
$$
As for $c \ell_{\ast}$, we consider $i \colon X \rightarrow W$ and $i' \colon X' \rightarrow W'$, and an isomorphism $\Phi \colon W \stackrel{\cong}{\longrightarrow} W'$ that restricts to an isomorphism $\Phi_{0} \colon X \stackrel{\cong}{\longrightarrow} X'$.
Invariance $\Phi_{0 \ast}\td_{\ast}(X) = \td_{\ast}(X')$ under algebraic isomorphisms $\Phi_{0} \colon X \stackrel{\cong}{\rightarrow} X'$ follows from \cite[p. 360, Example 18.3.3]{fultonintth}.
Hence, we obtain
$$
\Phi_{\ast}c \ell_{2\ast}(i) = \Phi_{\ast}i_{\ast}\td_{\ast}(X) = i_{\ast}'\Phi_{0\ast}\td_{\ast}(X) = i_{\ast}'\td_{\ast}(X') = c \ell_{2\ast}(i').
$$
To verify axiom (\ref{axiom locality}), we consider $i \colon X \rightarrow W$ and $f \colon M \rightarrow W$ such that $X \subset M$.
Then, the inclusion $i^{M} := i| \colon X \rightarrow M$ satisfies $f \circ i^{M} = i$, and we obtain
$$
f_{\ast}c \ell_{2 \ast}(i^{M}) = f_{\ast}i^{M}_{\ast}\td_{\ast}(X) = i_{\ast}\td_{\ast}(X) = c \ell_{2\ast}(i).
$$
Finally, to show axiom (\ref{axiom gysin}), let us proceed as in the proof of \Cref{proposition it class is l type characteristic class} and call closed irreducible subvarieties $Z, Z' \subset W$ of a smooth variety $W$ \emph{$\mathcal{X}_{CM}$-transverse} if $Z$ and $Z'$ are simultaneously complex algebraic Whitney transverse (that is, they admit complex algebraic Whitney stratifications such that every stratum of $Z$ is transverse to every stratum of $Z'$ as smooth submanifolds of $W$), generically transverse (see e.g. \cite[Section 3]{bw}), and Tor-independent (see \Cref{definition tor independence}) in $W$.
Now, consider an inclusion $i \colon X \rightarrow W$ in $\mathcal{X}_{CM}$ and an inclusion $f \colon M \rightarrow W$ (of a smooth closed subvariety $M \subset W$ in a smooth variety $W$) such that $M$ is irreducible, and $M$ and $X$ are $\mathcal{X}_{CM}$-transverse in $W$.
Then, \Cref{proposition transversality implies tightness} implies that the embedding $g \colon Y \hookrightarrow X$ of the compact subvariety $Y = X \cap M \subset X$ is tight.
For the regular closed embedding $g \colon Y \hookrightarrow X$ with algebraic normal bundle $N = N_{Y}X$, we have the Gysin restriction formula $g^{!}_{\alg}\td_{\ast}(X) = \td^{\ast}(N) \cap \td_{\ast}(Y)$ on Chow homology (see \cite[p. 361, Example 18.3.5]{fultonintth}).
The latter is a direct consequence of the Verdier-Riemann-Roch formula for the Todd class transformation $\tau_{\ast}$, which was conjectured by Baum-Fulton-MacPherson in \cite[p. 137]{bfm}, and proved by Verdier \cite[p. 214, Theorem 7.1]{verdierintcompl} (see also \cite[p. 349, Theorem 18.2(3)]{fultonintth}).
By invoking the cycle map $\operatorname{cl} \colon A_{\ast} (X)\otimes \rat \rightarrow H^\BM_{2*} (X)\otimes \rat$, we obtain
\begin{align*}
g^{!}_{\alg}\td_{\ast}(X) &= g^{!}_{\alg}\operatorname{cl}( \td_{\ast}(X)) = \operatorname{cl} g^{!}_{\alg}\td_{\ast}(X) \\
&= \operatorname{cl}(\td^{\ast}(N) \cap \td_{\ast}(Y)) = \td^{\ast}(N) \cap \operatorname{cl}(\td_{\ast}(Y)) = \td^{\ast}(N) \cap \td_{\ast}(Y),
\end{align*}
where we used that, according to Verdier \cite[p. 222, 9.2.1]{verdierintcompl}, the algebraic Gysin map of a closed regular embedding commutes with the cycle map (see diagram (\ref{dia.verdiershowbmcyclemapgysin})), and that the cycle map $\cl$ and the cap product with Chern classes are compatible by \cite[p. 374, Prop. 19.1.2]{fultonintth} (where note that the Todd class of a complex vector bundle is a rational polynomial in Chern classes, see e.g. \cite[p. 56, Example 3.2.4]{fultonintth}).
Since $i \in \mathcal{X}_{CM}$ and $M$ and $X$ are $\mathcal{X}_{CM}$-transverse in $W$, \Cref{alg and top gysin} and \Cref{top gysin borel moore and singular coincide} (where we may use the natural identification of Borel-Moore homology and singular homology because $X$ and $Y$ are compact) imply that the algebraic Gysin map $g^{!}_{\alg} \colon H_{\ast}(X; \mathbb{Q}) \rightarrow H_{\ast}(Y; \mathbb{Q})$ coincides with the topological Gysin map $g^{!}_{\operatorname{top}} \colon H_{\ast}(X; \mathbb{Q}) \rightarrow H_{\ast}(Y; \mathbb{Q})$ on all fundamental classes $[Z]_{X}$ of closed irreducible subvarieties $Z \subset X$.
As $\td_{\ast}(X) \in H_{\ast}(X; \mathbb{Q})$ is an algebraic cycle according to \Cref{rem.bfmtoddtochow}, we obtain
$$
g^{!}_{\operatorname{top}}\td_{\ast}(X) = g^{!}_{\alg}\td_{\ast}(X).
$$
Since the embedding $g \colon Y \hookrightarrow X$ is tight, we know that the underlying inclusion $Y \subset X$ is topologically normally nonsingular with topological normal bundle $\nu$ isomorphic (as a topological vector bundle) to the underlying topological vector bundle of the algebraic normal bundle $N = N_{Y}X$.
Next, recall from \Cref{normally nonsingular inclusions induced by transverse intersections} that the inclusion $g \colon Y \hookrightarrow X$ is normally nonsingular with topological normal bundle $\nu = j^{\ast} \nu_{f}$ given by the restriction under the inclusion $j \colon Y \rightarrow M$ of the normal bundle $\nu_{f}$ of $M$ in $W$, which is the the underlying topological vector bundle of the algebraic normal bundle $N_{f} = N_{M}W$.
Using the base change $f^{!}_{\operatorname{top}}i_{\ast} = j_{\ast}g^{!}_{\operatorname{top}}$ for topological Gysin maps (see \cite[Proposition 2.4]{bw}), as well as $\td^{\ast}(N) = \td^{\ast}(j^{\ast} N_{f}) = j^{\ast} \td^{\ast}(N_{f})$, we conclude that

\begin{align*}
f^{!}_{\operatorname{top}} c \ell_{2 \ast}(i) &= f^{!}_{\operatorname{top}} i_{\ast}\td_{\ast}(X) = j_{\ast}g^{!}_{\operatorname{top}} \td_{\ast}(X) = j_{\ast}g^{!}_{\alg} \td_{\ast}(X) = j_{\ast}(\td^{\ast}(N) \cap \td_{\ast}(Y)) \\
&= j_{\ast}(j^{\ast} \td^{\ast}(N_{f}) \cap \td_{\ast}(Y)) = \td^{\ast}(N_{f}) \cap j_{\ast}\td_{\ast}(Y) = c \ell^{\ast}(f) \cap c \ell_{2 \ast}(j).
\end{align*}

This completes the proof of \Cref{proposition todd class is l type characteristic class}.
\end{proof}

\end{document}